\documentclass[a4paper]{article}
\usepackage[all]{xy}\usepackage[latin1]{inputenc}        
\usepackage[dvips]{graphics,graphicx}
\usepackage{amsfonts,amssymb,amsmath,color,mathrsfs, amstext}
\usepackage{amsbsy, amsopn, amscd, amsxtra, amsthm,authblk}
\usepackage{enumerate,algorithmicx,algorithm}
\usepackage{algpseudocode}
\usepackage{upref}
\usepackage{geometry}
\usepackage{amsthm,amsmath,amssymb}
\usepackage{relsize}
\usepackage{mathrsfs}
\usepackage{bm}
\usepackage{exscale}
\usepackage{graphicx}
\usepackage{tabularx}
\usepackage{threeparttable,multirow,dcolumn,booktabs}
\usepackage{algorithmicx,algorithm}
\usepackage{tabularx}
\usepackage{caption}
\usepackage{float}
\usepackage{bm}
\usepackage{caption}
\usepackage{yhmath}
\usepackage{booktabs}
\usepackage{subcaption}
\usepackage{multirow}
\usepackage{makecell}
\usepackage[colorlinks,
            linkcolor=red,
            anchorcolor=red,
            citecolor=red
            ]{hyperref}

\numberwithin{equation}{section}

\numberwithin{equation}{section}

\DeclareMathOperator{\erf}{erf}
\DeclareMathOperator{\erfc}{erfc}

\def\Real{\mathrm{Re}}

\newtheorem{theorem}{Theorem}
\newtheorem{lemma}[theorem]{Lemma}
\newtheorem{proposition}[theorem]{Proposition}
\newtheorem{assumption}[theorem]{Assumption}
\newtheorem{corollary}[theorem]{Corollary}

\numberwithin{equation}{section}

\begin{document}

\title{Error estimate of the u-series method for molecular dynamics simulations}

\author[1,2]{Jiuyang Liang\thanks{liangjiuyang@sjtu.edu.cn}}
\author[1,3]{Zhenli Xu\thanks{xuzl@sjtu.edu.cn}}
\author[1]{Qi Zhou\thanks{zhouqi1729@sjtu.edu.cn}}

\affil[1]{School of Mathematical Sciences, Shanghai Jiao Tong University, Shanghai, 200240, P. R. China}

\affil[2]{Center for Computational Mathematics, Flatiron Institute, Simons Foundation, New York, 10010, USA}

\affil[3]{MOE-LSC, CMA-Shanghai and Shanghai Center for Applied Mathematics,  Shanghai Jiao Tong University, Shanghai, 200240, P. R. China}

\date{}
\maketitle

\begin{abstract}
This paper provides an error estimate for the u-series method of the Coulomb interaction in molecular dynamics simulations. We show that the number of truncated Gaussians $M$ in the u-series and the base of interpolation nodes $b$ in the bilateral serial approximation are two key parameters for the algorithm accuracy, and that the errors converge as $\mathcal{O}(b^{-M})$ for the energy and $\mathcal{O}(b^{-3M})$ for the force. Error bounds due to numerical quadrature and cutoff in both the electrostatic energy and forces are obtained. Closed-form formulae are also provided, which are useful in the parameter setup for simulations under a given accuracy. The results are verified by analyzing the errors of two practical systems.


\end{abstract}

\section{Introduction}

Molecular dynamics (MD) simulation is one of popular numerical tools in many-body particle systems and has been applied to various  areas such as biophysics, chemical physics and soft materials \cite{karplus1990molecular, french2010long,Scott2018Neuron,Allen2017ComputerLiquids}. By solving the Newton's equation for each particle that interacts with all other particles, MD is able to capture kinetic and thermodynamic quantities of the system by the ensemble average of particle configurations \cite{Frenkel2001Understanding}. MD with the interacting force field is a simple model whereas it has a dramatically expensive computational cost which is far beyond the computing power of a common personal computer for a typical simulation with mediate system size. The computational bottleneck in MD is the long-range nature of electrostatic interactions as well as its scalability in parallel computing. Popularly used methods are based on Ewald decomposition of Coulomb kernel $1/r$ \cite{Ewald1921AnnPhys}, such that the long-range component can be handled in Fourier space either by lattice summation methods such as the particle mesh Ewald \cite{Darden1993JCP} and particle-particle particle-mesh Ewald \cite{Hockney1988Computer}, or the random batch Ewald method \cite{Jin2020SISC}. The lattice summation methods have $\mathcal{O}(N\log N)$ complexity due to the use of fast Fourier transform (FFT) acceleration \cite{cooley1965algorithm}, and the random batch Ewald achieves $\mathcal{O}(N)$ complexity due to the mini-batch idea for the discrete Fourier transform. Other feasible methods include the treecode algorithm \cite{Barnes1986Nature}, the fast multipole method \cite{greengard1987fast}, 
the multigrid method \cite{trottenberg2000multigrid}, and the Maxwell-equation MD method \cite{maggs2002local}. Most of these methods have intensive global communications which causes low parallel scalability in large-scale simulations \cite{Arnold2013PRE}, and the algorithm development of electrostatics remains an attracting field.

The u-series method \cite{DEShaw2020JCP} is a recently developed algorithm for electrostatic calculations, which has been successfully implemented into Anton 3,
one of the most famous supercomputers specially designed for MD \cite{shaw2021anton}. Compared to the Ewald decomposition, the u-series gives
a far-field sum-of-Gaussians (SOG) approximation to Coulomb potential, which has many advantages. First, the SOG decomposition is able to conserve
the smoothness of any order at the cutoff point, providing higher accuracy under the same cost. Second, the separability of the SOG
is beneficial to save half of the sequential communication rounds for the FFT use. Finally, the algorithm achieves a trade-off among communication latency,
communication bandwidth and computation cost, which is advantageous for high precision parallel computation.
The u-series algorithm has been successfully applied to simulations of many popular problems such as the acceleration of the Covid-19 research \cite{malone2021structural,chen2022ensemble,padhi2021accelerating}. The extension of the u-series to the random batch SOG method was also reported recently~\cite{Liang2023SISC}, which accelerates the long-range calculation in the Fourier space by importance sampling instead of the use of FFT.

The u-series algorithm is practically effective and accurate, but it was pointed out \cite{DEShaw2020JCP} that a rigorous estimate of the convergence speed
due to the near-field cutoff remains unsolved. This results in a difficulty in the optimal parameter choice for different scales of systems for a given tolerance in MD simulations.
This paper fixes this gap and develops a rigorous error estimate on the electrostatic energy and forces for the u-series decomposition.
We achieve the proof of the error estimate by developing an approach of estimating
each Fourier mode of the near-field error.
We show that the decay rate is controlled by $\mathcal{O}(b^{-M})$ for the total energy and $\mathcal{O}(b^{-3M})$ for the forces,
where $b$ is the base of interpolation nodes for the SOG decomposition and $M$ is the number of Gaussians. The use of larger $b$ leads to the increase of the lower bound of the error. Moreover, the closed-form formulae for the approximate error are obtained.
Our error estimate is promising to offer a criterion of parameter setup in practical MD simulations.

The rest part of the paper is organized as follows. In Section \ref{useries}, we overview the u-series decomposition and present the main results on the near-field error estimation of $C^0$-continuous u-series. Section~\ref{proof} gives the detailed proof for the theorem in the main results. In Section~\ref{subsec::extension}, we extend the estimate to $C^1$-continuous u-series and discuss conditional convergence of u-series-based calculations. Section \ref{sec::closed} develops closed-form formulae of cutoff errors of energy and force for the use  parameter optimization. Section \ref{example} contains numerical validation on the performance of apriori parameter selection by using the theoretical results. Conclusions are made in Section \ref{conclusion}.

\section{Error estimate of the u-series decomposition}
\label{useries}

\subsection{U-series decomposition for Coulomb kernel}\label{subsec::u-series}
Consider a charged system of $N$ particles inside a cubic domain $\Omega$ with length $L$. The locations and charges of these particles are  $\{\bm{r}_i, q_i\}$, $i=1,2,\cdots,N$. Suppose that the domain is specified with 3D periodic boundary condition, that is, $\Omega$ as well as particles in it are replicated in all three directions in order to mimic the bulk environment. Assume that the system has charge neutrality condition $\sum_{i=1}^{N}q_i=0$. Then, the electrostatic potential on the $i$th particle is expressed as
\begin{equation}\label{solutionperiodic}
\Phi_i=\sum_{j=1}^{N}\sum_{\bm{n}\in\mathbb{Z}^{3}}~ '~\dfrac{q_j}{\left|\bm{r}_j-\bm{r}_i+\bm{n}L\right|},
\end{equation}
where the prime of the summation represents the ignoration of the case $i=j$ and $\bm{n}=\bm{0}$. The total energy of this system is given by
\begin{equation}
U=\frac{1}{2}\sum_{i=1}^{N}q_i\Phi_i,
\end{equation}
where the appearance of coefficient $1/2$ is to eliminate the double calculation of interactions. Furthermore, one can take its negative derivative with respect to location vector $\bm{r}_i$ to obtain the corresponding interaction force on the $i$th particle.

The calculation of series \eqref{solutionperiodic} is one of the most significant bottlenecks in MD simulation. This series converges conditionally \cite{Frenkel2001Understanding} since the Coulomb kernel $1/r$ has the long-range nature and a direct cutoff scheme is less accurate. Moreover, the $1/r$ kernel has a singularity at the origin, making it difficult to apply analytical methods such as Fourier transform for fast calculation.
The classical Ewald splitting \cite{Ewald1921AnnPhys} solves the problem by decomposing the Coulomb kernel into two parts: one is the near-field part which is fast decaying but with a singularity, and the other one is the smooth far-field part with long-range nature, namely,
\begin{equation}\label{splitting}
\dfrac{1}{r} =\frac{\erfc(\alpha r)}{r}+\frac{\erf(\alpha r)}{r},
\end{equation}
where $\erf(\cdot)$ is the error function, $\erfc(\cdot)$ is its complementary, and $\alpha$ is a positive parameter. For the near part, a cutoff radius $r_c$ is introduced and the influence of all particle pairs beyond this distance can be ignored. For the far part, it is smooth and varies gently, and decays rapidly in Fourier space (or $k$-space). Therefore, one can utilize some efficient techniques such as FFT and sampling methods to deal with this component.

The u-series approximation, however, is derived from a special SOG decomposition \cite{greengard2018anisotropic,exl2016accurate}. To describe the algorithm, we follow the work of Predescu {\it et al.} \cite{DEShaw2020JCP} and start from the integral identity
\begin{equation}\label{GammaExpansion}
\dfrac{1}{r^{2\beta}}=\dfrac{1}{\Gamma(\beta)}\int_{-\infty}^{\infty}e^{-e^tr^2+\beta t}dt,
\end{equation}
where $\Gamma(\cdot)$ represents the Gamma function. For the Coulomb kernel, one has $\beta=1/2$. Discretizing Eq.~\eqref{GammaExpansion} using the trapezoidal rule with uniformly spaced nodes $x_{\ell}=t_0+\ell h $, where $t_0=-\log 2\sigma^2$ and $h=2\log b$, yields a bilateral series approximation~\cite{beylkin2005approximation,beylkin2010approximation} of the Coulomb kernel, expressed as follows:
\begin{equation}\label{BSA}
\dfrac{1}{r}\approx\dfrac{2\log b}{\sqrt{2\pi\sigma^2}}\sum_{\ell=-\infty}^{\infty}\dfrac{1}{b^{\ell}}\exp\left[-\frac{1}{2}\left(\frac{r}{b^{\ell}\sigma}\right)^2\right].
\end{equation}
Clearly, it is a series of Gaussians. Here $b>1$ is the base of interpolation nodes, and $\sigma$ dominates the bandwidth of each Gaussian. Under the $b\rightarrow 1$ limit, one obtains the asymptotic error bound of the approximation
\begin{equation}\label{eq::pointwiseerror}
\left|1-\dfrac{2r\log b}{\sqrt{2\pi\sigma^2}}\sum_{\ell=-\infty}^{\infty}\dfrac{1}{b^{\ell}}\exp\left[-\frac{1}{2}\left(\frac{r}{b^{\ell}\sigma}\right)^2\right]\right|\lesssim 2\sqrt{2}\exp\left(-\dfrac{\pi^2}{2\log b}\right)
\end{equation}
for all $r>0$.

The so-named u-series \cite{DEShaw2020JCP} is due to the uniform error bound given by Eq.~\eqref{eq::pointwiseerror}. It modifies the bilateral series approximation such that the Coulomb kernel is decomposed as $1/r=\mathcal{F}_{b}^{\sigma}(r)+\mathcal{N}_{b}^{\sigma}(r)$, where the far part includes the long-range components of $\ell\geq0$ with ignorance of $\ell>M$ terms,  expressed by
\begin{equation}\label{eq::SOGField}
\mathcal{F}_{b}^{\sigma}(r)=\sum_{{\ell}=0}^{M}w_{\ell} e^{-r^2/s_{{\ell}}^2}
\end{equation}
with coefficients $
w_{\ell}=(\pi/2)^{-1/2}b^{-\ell}\sigma^{-1}\log b, \text{ and } s_{\ell}=\sqrt{2}b^{\ell}\sigma.	
$
For the near-field interaction, the complementary of the far part with radius cutoff at $r=r_c$ is introduced, that is,
\begin{equation}\label{eq::SOGDEcomp}
\mathcal{N}_{b}^{\sigma}(r)=\begin{cases}
1/r-\mathcal{F}_{b}^{\sigma}(r),\quad\text{if}~r<r_c\\\\0,\qquad\qquad\quad\,\,\, \text{if}~r\geq r_c.\end{cases}\quad\,
\end{equation}
In order to maintain the continuity of the near-field interaction $\mathcal{N}_{b}^{\sigma}(r)$, the cutoff radius $r_c$ should be set as the smallest root of $C^0$-continuous equation 
\begin{equation}
\label{eq::C0}
    r\mathcal{F}_{b}^{\sigma}(r)-1=0.
\end{equation} 
The smoothness of the u-series at $r_c$ can be improved by finely tuning parameters $w_{\ell}$ and $s_{\ell}$ of the far part. For instance, a $C^1$-continuity can be achieved by varying the coefficient of the narrowest Gaussian as $w_0\rightarrow \omega w_0$
with $\omega$ being a parameter, such that $r_c$ and $\omega$ are determined by satisfying both the $C^0$-continuity condition (Eq.~\eqref{eq::C0}) and the $C^1$-continuity condition
\begin{equation}
\label{eq::C1}
-\frac{1}{r_c^2}-\frac{d}{dr}\mathcal{F}_{b}^{\sigma}(r){\Big{|}}_{r=r_c}=0.
\end{equation}
Adjusting the parameters defining the narrowest Gaussian is necessary to prevent significant errors beyond the cutoff, as is suggested in~\cite{DEShaw2020JCP}.


The advantages of such a u-series decomposition are twofold. First, the potential is exact up to the cutoff radius and it is continuous at the cutoff point. Second, a Gaussian function is separable, i.e., it can be written as the product of one-dimensional functions. The separability of the Gaussian can be used to save half of the sequential communication rounds for the FFT acceleration~\cite{DEShaw2020JCP}. Due to these nice features, the u-series can produce the accuracy of the Ewald decomposition with a reduced computational effort, attracting much attention in the field of MD simulations.

\subsection{Main results of the convergence rate}\label{subsec::mainresult}
We study the contribution of near-field interactions beyond the cutoff distance in the u-series decomposition, obtaining the error estimate in the energy/force calculations. This section focuses on the $C^0$-continuous case and provides the main results. The detailed proof will be provided in Section~\ref{proof}. Extensions to $C^{1}$-continuous and higher continuity cases are discussed in Section~\ref{subsec::extension}.

By the u-series decomposition of the Coulomb kernel, the potential given by Eq.~\eqref{solutionperiodic} is approximated via 
\begin{equation}
\Phi_{i}\approx\Phi_{i,\text{u-series}}=\sum_{j=1}^{N}\sum_{\bm{n}\in\mathbb{Z}^{3}}~ '~q_j{\big[}\mathcal{F}_{b}^{\sigma}(\left|\bm{r}_j-\bm{r}_i+\bm{n}L\right|)+\mathcal{N}_{b}^{\sigma}(\left|\bm{r}_j-\bm{r}_i+\bm{n}L\right|){\big]}.
\end{equation}
One can write the error in the form as
\begin{equation}
\label{eq::error}
\Phi_{i,\text{err}}=\Phi_i-\Phi_{i,\text{u-series}}=\sum_{j=1}^{N}\sum_{\bm{n}\in \mathbb{Z}^3}q_jK(|\bm{r}_j-\bm{r}_i+\bm{n}L|),
\end{equation}
where the kernel function is
\begin{equation} \label{kernel:K}
K(r)=\left(\frac{1}{r}-\sum_{\ell=0}^{M}w_\ell e^{-r^2/s_\ell^2}\right) H(r-r_c)
\end{equation}
and $H(x)$ denotes the Heaviside step function and one has $H(x)=1$ for $x\geq0$ and $0$ otherwise. The errors in electrostatic energy and force can be then written by
\begin{equation} \label{4.2eq::Uerrferr}
U_{\text{err}}=\dfrac{1}{2}\sum_{i=1}^{N}q_i\Phi_{i,\text{err}},\quad \text{and}\quad\bm{F}_{\text{err}}(\bm{r}_i)=-\dfrac{\partial U_{\text{err}}}{\partial \bm{r}_i},
\end{equation}
respectively. The main objective of this paper is to estimate errors of both $U_{\text{err}}$ and $\bm{F}_{\text{err}}(\bm{r}_i)$.

The error estimate for Ewald-type methods have been well established \cite{kolafa1992cutoff,deserno1998mesh,wang2001estimate,wang2010optimizing}, providing a profound understanding of the accuracy of force calculation and introducing parameter tuning algorithms that boost the efficiency of the computation with a good control of the error. Unfortunately,  techniques of these works cannot be directly applied for estimating the error by the u-series method. This is mainly because that the long-range Coulomb kernel $1/r$ is still involved in $K(r)$ and its pairwise sum is only conditionally convergent. As a result, the discrete series $\Phi_i$ cannot be safely approximated by the corresponding continuous integral under the homogeneity assumption (i.e., the charges are distributed randomly for $r>r_c$). 

Before presenting the main results, one introduces Assumption \ref{ass:1} as an apriori knowledge of the relationship between the bandwidths of the SOG expansion and other parameters.

\begin{assumption}\label{ass:1}
	The bandwidth of the $(M+1)$th Gaussian in the bilateral series approximation is assumed to be much larger than the box size; i.e., one has the condition $s_{M+1}\gg L\geq 2r_c$. Here $L\geq 2r_c$ is due to the minimum image convention. Moreover, the bandwidth of the $-1$st Gaussian is assumed to be less than the cutoff radius, i.e.,  $s_{-1}<r_{c}$.
\end{assumption}

Assumption \ref{ass:1} is a natural setup for practical systems. Actually, if $s_{M+1}$ is small such that it is comparable to $L$, then the numerical error will be big due to the significant truncated Gaussians and cannot be ignored. Similarly, the error due to the neglection of the $(-1)$-st Gaussian will become at least $O(1)$ if  $s_{-1}>r_{c}$. Under Assumption~\ref{ass:1}, we have Theorem \ref{thm::thm4}, which claims that both $U_{\text{err}}$ and $\bm{F}_{\text{err}}(\bm{r}_i)$ have a uniform bound with the increase of the u-series truncation term $M$.

\begin{theorem}\label{thm::thm4}
For $C^0$-continuous u-series, $U_{\emph{err}}$ and $\bm{F}_{\emph{err}}(\bm{r}_i)$ hold the following estimates,
	\begin{equation}\label{eq::errorfinal}
	\begin{split}
	|U_{\emph{err}}|&\simeq O\left((\log b)^{-3/2}e^{-\pi^2/2\log b}+b^{-M}+w_{-1}e^{-r_c^2/s_{-1}^2}\right),\\
	|\bm{F}_{\emph{err}}(\bm{r}_i)|&\simeq O\left((\log b)^{-3/2}e^{-\pi^2/2\log b}+b^{-3M}+w_{-1}(s_{-1})^{-2}e^{-r_c^2/s_{-1}^2}\right),
	\end{split}
	\end{equation}
	where $\simeq$ indicates ``asymptotically equal'' as $b\rightarrow 1$.
\end{theorem}
From Theorem \ref{thm::thm4}, it is observed that the decay rate with respect to the number of truncated Gaussians is controlled by $O(1/b^M)$ for the energy and $O(1/b^{3M})$ for the forces. The other two terms in Eq.~\eqref{eq::errorfinal} determine the lower bound on the errors. Intuitively, the first term in Eq.~\eqref{eq::errorfinal} represents the ``aliasing'' error arising from the trapezoidal rule, depending solely on $b$. The second and third terms are resulted from neglecting Gaussians with $\ell>M$ and $\ell<0$, respectively. Note that the bilateral series approximation gives
\begin{equation}\label{eq::pointwise}
    \frac{1}{r}\approx \sum_{\ell=-\infty}^{\infty}w_\ell e^{-r^2/s_\ell^2}\quad \text{and} \quad \nabla_{\bm{r}}\left(\frac{1}{r}\right)=-\frac{\bm{r}}{r^3}\approx 2\bm{r}\sum_{\ell=-\infty}^{\infty}\frac{w_\ell}{s_\ell^2}e^{-r^2/s_\ell^2} \quad \text{as}\quad b\rightarrow 1
\end{equation}
for $r\in(0,\infty)$. Truncating the upper limit of the two sums of Gaussians in Eq.~\eqref{eq::pointwise} at $\ell=M$ will lead to the effective range of approximation being truncated to $(0,r_{\text{max}}]$ with $r_{\text{max}}\sim s_{M}$. Given $w_{M}\sim 1/b^{M}$ and $s_{M}\sim b^M$, the approximation errors of $1/r$ and its gradient around $r=r_{\text{max}}$ are $O(1/b^M)$ and $O(1/b^{3M})$, respectively. Actually, Theorem~\ref{thm::thm4} indicates that these pointwise behaviors will be inherited in the lattice summation, with decay rates of $1/b^{M}$ and $1/b^{3M}$ for energy and forces, respectively. Similarly, truncating the lower limit in Eq.~\eqref{eq::pointwise} at $\ell=0$ leads to the effective range of approximation being further truncated to $[r_{\min},r_{\max}]$ with $r_{\min}\sim s_{-1}<r_c$. The dominant contribution naturally originates from the Gaussian with $\ell=-1$ and neighbors situated around the cutoff distance $r_c$.

Theorem~\ref{thm::thm4} also leads to a simple rule, as presented in Corollary~\ref{coroll:parameter}, to set $b$ and $M$ to guarantee a user-defined decomposition error of force, given its crucial role in MD. The derivation relies on Lemma~\ref{lem::estimate} which is a refinement of Lemma 14 in~\cite{beylkin2010approximation}. The estimate of $U_{\text{err}}$ should also be valuable for applications like Monte Carlo simulations.
\begin{lemma}\label{lem::estimate}
Let $x$ and $\varepsilon$ be two positive numbers. The inequality $x^{3/2}e^{-x}\leq \varepsilon$ holds if 
\begin{equation}
x\geq\left(\frac{3}{2}\log\frac{3}{2}+\log\varepsilon^{-1}\right)\frac{e}{e-1}.
\end{equation}
\end{lemma}
\begin{corollary}\label{coroll:parameter}
Suppose that an error tolerance $\varepsilon>0$ is given. By Lemma~\ref{lem::estimate}, one can set 
\begin{equation}
b\sim e^{\pi^2(e-1)/(2ep)}\quad\text{with}\quad p=\frac{3}{2}\log\frac{3}{2}+\log(3\pi^3)+\log (2^{-3/2}\varepsilon^{-1}),
\end{equation}
so that the aliasing error $\sim\varepsilon/3$. In addition, find proper $r_c\sim s_{-1}\sqrt{\log (3w_{-1}s_{-1}^2)+\log\varepsilon^{-1}}$ and $M\sim (3\log b)^{-1}\log (3\varepsilon^{-1})$, so that the error for dropping $\ell>M$ and $\ell<0$ Gaussians are both $\sim \varepsilon/3$. The error of $C^0$-continuous u-series is thus $\sim \varepsilon$ uniformly over all $r\in \Omega$. 
\end{corollary}

It is important to note that Theorem~\ref{thm::thm4} specifically shows the convergence rate of both energy and force. Thus, Corollary~\ref{coroll:parameter} provides a practical strategy. However, finely optimizing pertinent parameters in u-series MD simulations remains challenging due to the significant computational overhead associated with prefactors.
 
\section{Proof of the main results} \label{proof}
In this subsection, we present a detailed proof of Theorem~\ref{thm::thm4}. It is remarked that we do not use any mean-field assumptions in the proof. 

\subsection{Fourier spectral expansion and kernel splitting}
We consider to estimate the errors by the Fourier method. As usual, the tinfoil boundary condition is specified for $r\rightarrow \infty$ and thus the zero-frequency term is neglected~\cite{DEShaw2020JCP,hu2014infinite}. This allows us to seek for the bound of $\phi_{\text{err}}$ in Fourier space. Discussions on other infinite boundary conditions are provided in Section~\ref{subsec::extension}. Let us define the Fourier transform pairs as
\begin{equation}\label{eq::Fouriers}
\widetilde{f}(\boldsymbol{k}):=\int_{\Omega} f(\boldsymbol{r}) e^{-i \boldsymbol{k} \cdot \boldsymbol{r}} d \boldsymbol{r} \quad \text { and } \quad f(\boldsymbol{r})=\frac{1}{V} \sum_{\boldsymbol{k}} \widetilde{f}(\boldsymbol{k}) e^{i \boldsymbol{k} \cdot \boldsymbol{r}},
\end{equation}
$\text {with } \boldsymbol{k}=2 \pi\bm{m}/L \text { and } \boldsymbol{m}=\left(m_x, m_y, m_z\right) \in \mathbb{Z}^3 \text {. }$ The Fourier expansions of both $U_{\text{err}}$ and $\bm{F}_{\text{err}}(\bm{r}_i)$ are expressed in Proposition~\ref{prop::spectral}.

\begin{proposition}\label{prop::spectral}
Under the tinfoil boundary condition, $U_{\emph{err}}$ and $\bm{F}_{\emph{err}}(\bm{r}_i)$ have the following Fourier spectral expansion:
\begin{equation}
\label{eq::UerrFerr}
U_{\emph{err}}=\frac{1}{2V}\sum_{\bm{k}\neq \bm{0}}|\rho(\bm{k})|^2\widetilde{K}(k),\quad \text{and}\quad	\bm{F}_{\emph{err}}(\bm{r}_i)=\sum_{\bm{k}\neq \bm{0}}\frac{q_i\bm{k}}{V}\emph{Im}\left(e^{-i\bm{k}\cdot\bm{r}_i}\rho(\bm{k})\right)\widetilde{K}(k),
\end{equation}
where $k=|\bm{k}|$, $\widetilde{K}(k)$ is the Fourier transform of $K(r)$, and $\rho(\bm{k})$ is the structure factor expressed by
\begin{equation}
\rho(\bm{k})=\sum_{j=1}^{N}q_je^{i\bm{k}\cdot \bm{r}_j}.
\end{equation}
\end{proposition}
Section \ref{subsec::u-series} shows that the u-series decomposition is derived from applying the trapezoidal rule to the integral representation of $1/r$ and dropping unwanted terms. For the convenience of analysis, one can split the kernel function $K(r)$ in Eq.~\eqref{kernel:K}  into three components:
\begin{equation}
K(r)=T(r)+G_{\text{up}}(r)+G_{\text{down}}(r),
\end{equation}
where
\begin{equation}\label{eq::K1}
T(r)=\left(\frac{1}{r}-\sum_{\ell=-\infty}^{\infty}w_\ell e^{-r^2/s_\ell^2}\right)H(r-r_c),
\end{equation}
\begin{equation}
\label{eq::GupGdow} 
G_\textup{up}(r)=\sum_{\ell=M+1}^{\infty}w_\ell e^{-r^2/s_\ell^2}H(r-r_c), 
\end{equation}
and
\begin{equation}\label{eq::Gdowndefi}
G_\textup{down}(r)= \sum_{\ell=-\infty}^{-1}w_\ell e^{-r^2/s_\ell^2} H(r-r_c). 
\end{equation}
Here, $T(r)$ recovers the approximate error of the trapezoidal rule, and $G_{\text{up}}(r)$ and $G_{\text{down}}(r)$ arise due to the exclusion of Gaussians with $\ell>M$ and $\ell<0$ , respectively.

To give a feasible bound of $U_{\text{err}}$ and $\bm{F}_{\text{err}}(\bm{r}_i)$, one needs to derive analytic expressions of the Fourier transforms of these kernel functions. One first introduces Lemma \ref{lemma1} for the Fourier transform of a radially symmetric function \cite{stein2011fourier}.   
\begin{lemma}
	\label{lemma1}
	Assume that the Fourier transform of $f(\bm{x})$ exists. If $f(\bm{x})$ is a radially symmetric function in 3D, its Fourier transform is also radially symmetric, expressed by 
	\begin{equation}
	F(k)=4\pi\int_{0}^{\infty}\frac{\sin(kr)}{k}f(r)rdr.
	\end{equation}
\end{lemma}

By Lemma \ref{lemma1}, one can represent the 3D Fourier transforms of $T,G_{\text{up}}$ and $G_{\text{down}}$  as 1D integrals, that is,
\begin{equation}\label{eq::3.91}
\begin{aligned}
&\widetilde{T}(k)=\dfrac{4\pi}{k}\int_{r_c}^{\infty}\sin(kr)\left(1-r\sum_{\ell=-\infty}^{\infty}w_\ell e^{-r^2/s_\ell^2}\right)dr,\\
&\widetilde{G}_{\text{up}}(k)=\dfrac{4\pi}{k}\int_{r_c}^{\infty}\sin(kr)r\sum_{\ell=M+1}^{\infty}w_\ell e^{-r^2/s_\ell^2}dr,\\
&\widetilde{G}_{\text{down}}(k)=\dfrac{4\pi}{k}\int_{r_c}^{\infty}\sin(kr)r\sum_{\ell=-\infty}^{-1}w_\ell e^{-r^2/s_\ell^2}dr.
\end{aligned}
\end{equation}
An essential goal of this paper is to analyze the leading asymptotic orders of these oscillatory Fourier integrals as $b\rightarrow 1$, under the conditions stated in  Assumption~\ref{ass:1}. The results are outlined in Propositions~\ref{prop::prop1}-\ref{prop::prop3}, with detailed proofs provided in Appendices~\ref{pf::prop1}-\ref{pf::prop3}.

\begin{proposition}\label{prop::prop1}
If the tinfoil boundary condition is specified, $\widetilde{T}(k)$ can be expressed as
	\begin{equation}\label{eq::K1kk}
	\widetilde{T}(k)=\sum_{m\neq 0}\mathcal{C}(m)\widetilde{\phi}_m(k)
	\end{equation}
	where the summation over $m$ ranges over $\mathbb{Z}$ excluding $m=0$,
	\begin{equation}\label{eq::Cndef}
	\mathcal{C}(m)=-\dfrac{\Gamma\left((1-\alpha_m)/2\right)}{\sqrt{\pi}}e^{-\alpha_m \log\sqrt{2}\sigma}
	\end{equation}
	is an $m$-dependent constant with $\alpha_m=2m\pi i/\log b$ and $\Gamma(\cdot)$ being the Gamma function, and
	\begin{equation}\label{eq::Th1Ink}
	\widetilde{\phi}_m(k)=\frac{4\pi}{k^{2+\alpha_m}}\left[\mathcal{I}_m-\mathcal{J}_{m}(k)\right]
	\end{equation}
	with two components
	\begin{equation}
	\label{eq::Im}
	\mathcal{I}_m:=\cosh\left(\frac{m\pi^2}{\log b}\right)\Gamma\left(1+\alpha_m\right)
\quad	\text{and} \quad
	\mathcal{J}_{m}(k):=\int_{0}^{kr_c}\sin(x)x^{\alpha_m}dx.
	\end{equation} 
\end{proposition}
\begin{proposition}\label{prop::prop2}
	$\widetilde{G}_{\emph{up}}(k)$ can be expressed as
	\begin{equation}
	\begin{split}
	\widetilde{G}_{\emph{up}}(k)= 2\pi\log b\sum_{\ell=M+1}^{\infty}s_{\ell}^2e^{-\frac{s_{\ell}^2k^2}{4}}+\dfrac{4\pi}{k^2}\left[\frac{2\log b}{(b-1)\sqrt{\pi}s_M}+O\left(\frac{L^2}{s_M^3}\right)\right]\mathscr{P}(r_c,k)
	\end{split}
	\end{equation}
	where
	\begin{equation}\label{eq::Prck}
	\mathscr{P}(r_c,k):=\frac{\sin(kr_c)}{k}-r_c\cos(kr_c).	
	\end{equation}
\end{proposition}
\begin{proposition}\label{prop::prop3}
	$\widetilde{G}_{\emph{down}}(k)$ has the expression
	\begin{equation}
	\widetilde{G}_{\emph{down}}(k)=\frac{4\pi}{k^2}\sum_{\ell=-\infty}^{-1}w_{\ell}e^{-r_c^2/s_{\ell}^2}\beta_{\ell}(k),
	\end{equation}
	where
	\begin{equation}\label{eq::2.25}
	\beta_{\ell}(k):=\frac{(ks_{\ell})^2r_c\cos(kr_c)+2kr_c^2\sin(kr_c)}{(2r_c/s_{\ell})^2+(ks_{\ell})^2}\left[1+O\left(\frac{r_c}{s_{\ell}}\right)^{-2}\right].
	\end{equation}
 
\end{proposition}

Substituting these results into Eq.~\eqref{eq::UerrFerr} allows for establishing an error estimate for the u-series decomposition. Consider the error in electrostatic energy 
\begin{equation}\label{eq::Uerr}
\begin{split}
U_{\text{err}}=&\frac{1}{2V}\sum_{\bm{k}\neq \bm{0}}|\rho(\bm{k})|^2\left[\widetilde{T}(k)+\widetilde{G}_{\text{up}}(k)+\widetilde{G}_{\text{down}}(k)\right]\\
:=&E_{T}+E_{G}^{\text{up}}+E_{G}^{\text{down}}.
\end{split}
\end{equation}
Next, we will estimate these three terms separately.

\subsection{The estimate of \texorpdfstring{$E_{T}$}~}\label{subsec::1}
To estimate $E_{T}$, we introduce Lemma \ref{lemma:Gamma} for the Gamma function with complex variable. The proof of this lemma can be found in the Appendix of Ref. \cite{DEShaw2020JCP}.

\begin{lemma}\label{lemma:Gamma}
	Gamma function $\Gamma(\cdot)$ has the following asymptotic approximation
	\begin{equation}\label{eq::lemma3}
	\left|\Gamma(\alpha+i\beta)\right|\simeq (2\pi)^{1/2}(\alpha^2+\beta^2)^{\frac{2\alpha-1}{4}}e^{-\frac{\pi}{2}|\beta|}
	\end{equation}
	at the $\beta\rightarrow \infty$ limit.
\end{lemma}

Using Lemma \ref{lemma:Gamma} with $\alpha=1/2$ and $\beta=-m\pi/\log b$,
\begin{equation}\label{eq::Cn}
\left|\mathcal{C}(m)\right|= \dfrac{1}{\sqrt{\pi}}\left|\Gamma\left(\frac{1-\alpha_m}{2}\right)\right|\simeq \sqrt{2}e^{\frac{-m\pi^2}{2\log b}}
\end{equation}
holds for $b\rightarrow 1$. Combining Eqs.~\eqref{eq::Cn} and \eqref{eq::K1kk}, one has
\begin{equation}\label{eq::final1}
\begin{split}
E_{T}&=\frac{1}{2V}\sum_{\bm{k}\neq \bm{0}}|\rho(\bm{k})|^2\sum_{m\neq 0}\mathcal{C}(m)\widetilde{\phi}_m(k)\\
&\simeq \frac{2\sqrt{2}\pi}{V}\sum_{\bm{k}\neq \bm{0}}\sum_{m\neq 0}\dfrac{e^{\frac{-m\pi^2}{2\log b}}}{k^{2+\alpha_m}}|\rho(\bm{k})|^2\left[\mathcal{I}_m-\mathcal{J}_{m}(k)\right].
\end{split}
\end{equation}
Recall that the definitions of $\mathcal{I}_m$ and $\mathcal{J}_{m}(k)$ are given in Eq.~\eqref{eq::Im}. The convergence of series given in Eq.~\eqref{eq::final1} is of significance, and is proved in Theorem~\ref{the::conver}.
\begin{theorem}
    \label{the::conver}
The series representation of $E_T$ provided in Eq.~\eqref{eq::final1} converges under the condition of charge neutrality and the specified tinfoil boundary condition.
\end{theorem}
\begin{proof}
	Since $k^{-(2+\alpha_m)}|\rho(\bm{k})|^2\left[\mathcal{I}_m-\mathcal{J}_{m}(k)\right]$ has no singularity in the case of $k$ and $m\neq 0$, the convergence of series should be determined by the behavior of $k$ approaching infinity. Recalling Eq.~\eqref{eq::Im} and by applying Lemma~\ref{lemma:Gamma}, one has an asymptotic estimate of $\mathcal{I}_{m}$, 
\begin{equation}\label{eq::estimateIm}
\mathcal{I}_{m}\simeq \sqrt{\dfrac{\pi}{2}}\left(1+|\alpha_m|^2\right)^{\frac{1}{4}}.
\end{equation} 
Note that the integral term $\mathcal{J}_{m}(k)$ has a closed-form expression
\begin{equation}\label{eq::closed}
\mathcal{J}_{m}(k)=\frac{(kr_c)^{2+\alpha_m}}{2+\alpha_m}\mathscr{H}\left(1+\frac{\alpha_m}{2},\left[\frac{3}{2},2+\frac{\alpha_m}{2}\right],-\frac{(kr_c)^2}{4}\right),
\end{equation}
where $\mathscr{H}(\cdot,\cdot,\cdot)$ represents the generalized hypergeometric function \cite{abramowitz1964handbook}. By Eq.~\eqref{eq::estimateIm} and applying asymptotic analysis \cite{murray2012asymptotic,knottnerus1960approximation} to Eq.~\eqref{eq::closed}, one can find that as $k\rightarrow \infty$,
	\begin{equation}\label{eq:asympt}
	\frac{\mathcal{I}_m-\mathcal{J}_{m}(k)}{k^{\alpha_m}}=r_c^{\alpha_m}\left[\cos(kr_c)+\frac{\alpha_m}{kr_c}\sin(kr_c)+o\left(\frac{1}{k}\right)\right].
	\end{equation}
	By taking the first term and throwing out all prefactors related to $m$ (due to the rapidly decay of exponential factor), Theorem~\ref{the::conver} is equivalent to proving the convergence of
	\begin{equation}
	U_{> r_c}:=\frac{2\pi }{V}\sum_{\bm{k}\neq \bm{0}}\dfrac{\cos(kr_c)}{k^{2}}|\rho(\bm{k})|^2.
	\end{equation} 
	By simple calculation, one can find that the factor $k^{-2}\cos(kr_c)$ is exactly the three-dimensional generalized Fourier transform of $r^{-1}H(r-r_c)$. This implies that $U_{> r_c}$ is nothing but part of the electrostatic energy contributed from the pairs with the distance larger than the cutoff radius $r_c$. The convergences of $U_{> r_c}$ as well as $E_T$ are thus ensured by the charge neutrality condition and the tinfoil boundary condition. 

\end{proof}

We now proceed to estimate the leading order of $E_T$. We asymptotically approximate the sum over vectors $\bm{k}$ by the integral similar to the technique for the Ewald summation \cite{kolafa1992cutoff},
\begin{equation}\label{eq::approx}
\sum_{\bm{k}\neq \bm{0}}\simeq \dfrac{V}{(2\pi)^3}\int_{\frac{2\pi}{L}}^{\infty}k^2dk\int_{-1}^{1}d\cos\theta\int_{0}^{2\pi}d\varphi,
\end{equation}
where $(k, \theta, \varphi)$ are the spherical coordinates and $\simeq$ indicates asymptotically equal as $L\rightarrow\infty$.  When Eq.~\eqref{eq::approx} is applied to a Fourier series (say Eq.~\eqref{eq::final1}), one can choose a specific $(k, \theta, \varphi)$ so that the coordinate along $\cos\theta$ of $\bm{k}$ is in the direction of a specific vector $\bm{r}$, and $\bm{k}\cdot\bm{r}=kr\cos\theta$. Even for finite $L$, Eq.~\eqref{eq::approx} remains valid for our estimation, as it does not alter the leading decay order. Further discussion on this point is provided in Appendix~\ref{app::discussintegral}.

By Eq.~\eqref{eq::final1}, let us define
\begin{equation}
E_{T,m}(\bm{k}):=\dfrac{2\sqrt{2}\pi e^{\frac{-m\pi^2}{2\log b}}}{Vk^{2+\alpha_m}}|\rho(\bm{k})|^2\left[\mathcal{I}_m-\mathcal{J}_{m}(k)\right]
\end{equation}
so that $E_T=\sum\limits_{\bm{k}\neq \bm{0}}\sum\limits_{m\neq 0}E_{T,m}(\bm{k})$. By performing the integral approximation, one has
\begin{equation}\label{Uerr1}
\begin{split}
\sum_{\bm{k}\neq \bm{0}}E_{T,m}(\bm{k})&\simeq\sqrt{2}e^{\frac{-m\pi^2}{2\log b}}\left(\int_{\frac{2\pi}{L}}^{\frac{|\alpha_{m}|}{r_c}}+\int_{\frac{|\alpha_{m}|}{r_{c}}}^{\infty}\right)(\pi k^{\alpha_m})^{-1}\sum_{i,j}q_iq_j\frac{\sin(kr_{ij})}{kr_{ij}}\left[\mathcal{I}_m-\mathcal{J}_{m}(k)\right]dk\\
&:=\sqrt{2}e^{\frac{-m\pi^2}{2\log b}}(\mathcal{E}_{T,m}+\mathcal{R}_{T,m}).
\end{split}
\end{equation}
The upper limit of integral in $\mathcal{E}_{T,m}$ is $|\alpha_m|/r_c$ with $|\alpha_m|\sim (\log b)^{-1}$, so we can consider the asymptotic behaviors of $\mathcal{I}_m$ and $\mathcal{J}_m(k)$ as $b\rightarrow 1$. It is clear by Eq.~\eqref{eq::Im} that $\mathcal{I}_m$ is scaled as $O((\log b)^{-1/2})$. However, by applying the integral by parts to Eq.~\eqref{eq::Im}, the high oscillation factor $x^{\alpha_m}$ will lead an $(\alpha_m)^{-1}=O(\log b)$ decay order of $\mathcal{J}_m(k)$ for $k\in[2\pi/L,|\alpha_m|/r_c]$. Hence one has
\begin{equation}\label{eq::Etn2}
\mathcal{E}_{T,m}\simeq O(|\alpha_{m}|\mathcal{I}_{m})=O((\log b)^{-3/2}) \quad\text{as}\quad b\rightarrow 1.
\end{equation}
Next, the asymptotic expansion given by Eq.~\eqref{eq:asympt} is valid on the interval $[|\alpha_m|/r_c,\infty)$, so that the leading order of the remainder term $\mathcal{R}_{T,m}$ reads
\begin{equation}\label{eq::3.30}
\mathcal{R}_{T,m}\simeq \frac{r_c^{\alpha_m}}{\pi}\sum_{i,j}q_iq_j\int_{\frac{|\alpha_{m}|}{r_c}}^{\infty}\frac{\sin(kr_{ij})}{kr_{ij}}\cos(kr_c)dk= O\left(|\alpha_{m}|^{-1}\right)=O(\log b)\quad\text{as}\quad b\rightarrow 1.
\end{equation}
We also study the properties of $\mathcal{J}_{m}(k)$ by auxiliary numerical calculation. It is observed from Fig.~\ref{fig:IInk} that in typical cases there are many oscillations of $\mathcal{J}_{m}(k)$ around $\mathcal{I}_m$ for $k>|\alpha_m|/r_c$. When $k\leq|\alpha_m|/r_c$, the value of $\mathcal{I}_{m}(k)$ is significantly larger than that of $\mathcal{J}_{m}(k)$. This implies that $\mathcal{R}_{T,m}$ is a higher-order infinitesimal as $b\rightarrow 1$ in comparison to $\mathcal{E}_{T,m}$, in consistent with our theoretical analysis. Substituting Eqs.~\eqref{eq::Etn2} and \eqref{eq::3.30} into Eq.~\eqref{Uerr1}, one has
\begin{equation}\label{eq::Uerr1}
E_{T}=\sum_{\bm{k}\neq\bm{0}}\sum_{m\neq 0}E_{T,m}(\bm{k})\simeq O\left((\log b)^{-3/2}e^{\frac{-\pi^2}{2\log b}}\right)\quad\text{as}\quad b\rightarrow 1.
\end{equation}
If more precise value of $E_T$ needs to be calculated, the integral along with the costly hypergeometric function in Eq.~\eqref{Uerr1} can be numerically evaluated by applying suitable quadratures and interpolation methods. However, the total complexity for evaluating Eq.~\eqref{Uerr1} is still $O(N^2)$ due to the double summation.

\begin{figure*}[ht]	
	\centering
	\includegraphics[width=0.95\textwidth]{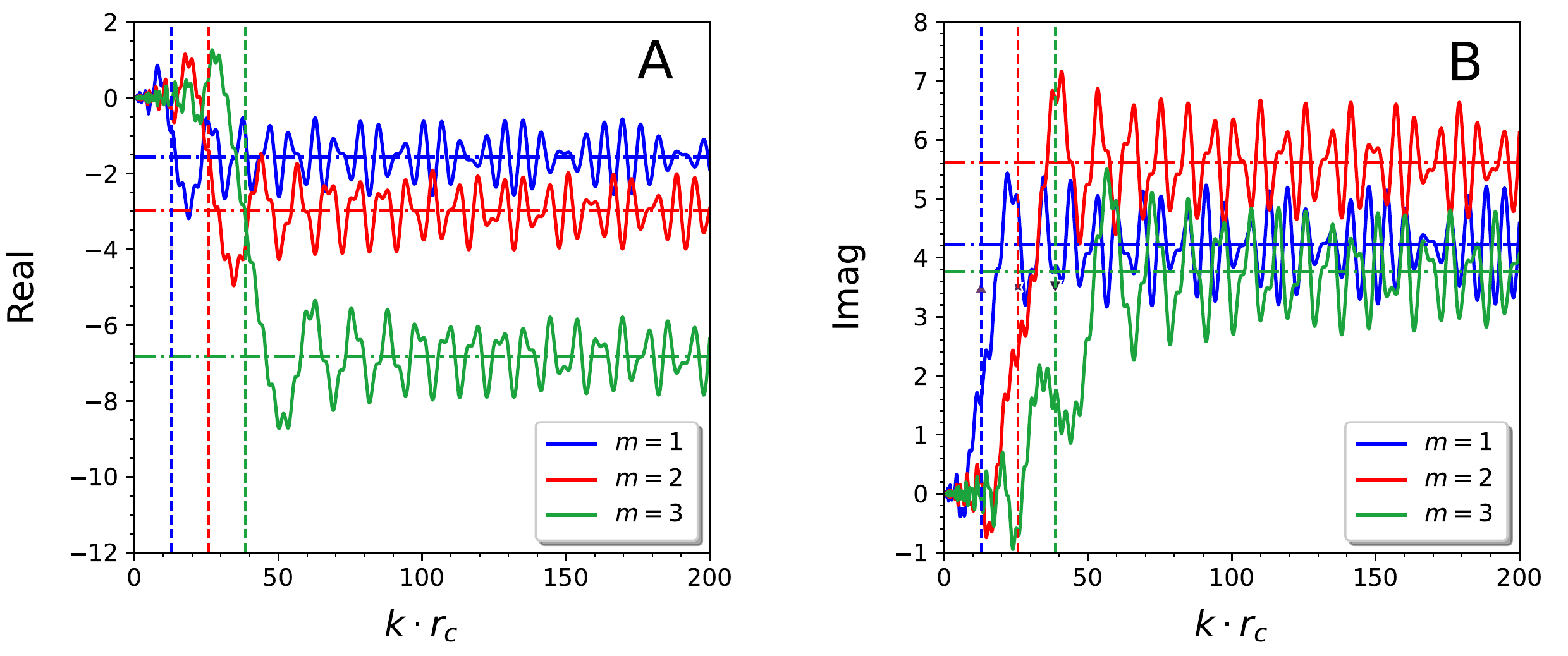}
	\caption{Values of (A) the real part and (B) the imaginary part of $\mathcal{J}_{m}(k)$ against the value of $kr_c$ for $m=1,2,3$. The dash-dotted lines in (A-B) indicate the real and imaginary parts of $\mathcal{I}_m$, respectively. The dash lines indicate $x\equiv |\alpha_m|$ for different $m$.}
	\label{fig:IInk}
\end{figure*}

\subsection{The estimate of \texorpdfstring{$E_{G}^{\mathrm{up}}$}~}\label{subsec::2}

Next, let us consider the estimation of
\begin{equation}
E_{G}^{\text{up}}=\frac{1}{2V}\sum_{\bm{k}\neq \bm{0}}|\rho(\bm{k})|^2\widetilde{G}_{\text{up}}(k).
\end{equation}
Utilizing Proposition \ref{prop::prop2} and employing the integral approximation Eq.~\eqref{eq::approx} again, one has
\begin{equation}\label{eq::final2}
\begin{split}
E_{G}^{\text{up}}
&\simeq \sum_{i,j}q_iq_j\int_{0}^{\infty}\dfrac{\sin(kr_{ij})}{\pi k r_{ij}}\left[\frac{k^2\log b}{2}\sum_{\ell=M+1}^{\infty} s_\ell^2e^{-s_\ell^2k^2/4} +\frac{2\log b}{\sqrt{\pi}(b-1)s_M }\mathscr{P}(r_c,k)\right]dk\\
&=\sum_{i,j}q_iq_j\left[\frac{\log b}{\sqrt{\pi}}\sum_{\ell=M+1}^{\infty}\frac{e^{-r_{ij}^2/s_{\ell}^2}}{s_\ell}+ \frac{\log b}{\pi(b-1)s_M}\left(\sqrt{\pi}\min\left\{\frac{ r_c}{r_{ij}},1\right\}-\frac{2r_c}{\sqrt{\pi}r_{ij}}\mathcal{T}(r_{ij},r_c)\right)\right]
\end{split}
\end{equation}
where the lower limit of integral starts from zero instead of $2\pi/L$ due to $\lim\limits_{\bm{k}\rightarrow \bm{0}}|\rho(\bm{k})|^2\widetilde{G}_{\text{up}}(k)=0$ resulting from the charge neutrality condition. Here,
\begin{equation}
\mathcal{T}(r_{ij},r_c):=\int_{0}^{\infty}\dfrac{\sin(kr_{ij})}{k}\cos(kr_c)dk=\begin{cases}\dfrac{\pi}{2},\quad&\text{for}~r_{ij}>r_c\\[1em]
\dfrac{\pi}{4},\quad&\text{for}~r_{ij}=r_c\\[1em]
0,\quad&\text{for}~r_{ij}<r_c\end{cases}
\end{equation}
represents a type of Borwein integral~\cite{borwein2001some}. Indeed, the above Borwein integral is undefined at $r_{ij}=r_c$, and it thus can be taken as any value. Here the half of $\pi/2$ is a general choice~\cite{katznelson2004introduction}. 
Considering that each component in Eq.~\eqref{eq::final2} is scaled as $O\left(1/s_M\right)$ and $s_{M}=O(b^M)$, it follows that $|E_{G}^{\text{up}}|\simeq O(1/b^M)$.


\subsection{The estimate of \texorpdfstring{$E_{G}^{\mathrm{down}}$}~}\label{subsec::3}
Consider
\begin{equation}
\label{eq::EGDown}
E_{G}^{\text{down}}=\frac{1}{2V}\sum_{\bm{k}\neq \bm{0}}|\rho(\bm{k})|^2\widetilde{G}_{\text{down}}(k).
\end{equation}
By Proposition \ref{prop::prop3} and applying the integral approximation Eq.~\eqref{eq::approx}, one similarly has
\begin{equation}\label{eq::3.35}
E_{G}^{\text{down}}\simeq \dfrac{1}{2\pi^2}\sum_{\ell=-\infty}^{-1}\left[w_{\ell}e^{-r_c^2/s_{\ell}^2}\sum_{i,j}q_iq_j\int_{0}^{\infty}\dfrac{\sin(kr_{ij})}{kr_{ij}}\beta_{\ell}(k)dk\right],
\end{equation}
where the integral can be explicitly computed by keeping the leading term,
\begin{equation}\label{eq::3.36}
\int_{0}^{\infty}\dfrac{\sin(kr_{ij})}{kr_{ij}}\beta_{\ell}(k)dk=\begin{cases}
0,\quad &r_{ij}\leq r_c,\\\\
\dfrac{\pi r_c}{2r_{ij}}e^{-2r_c(r_{ij}-r_c)/s_{\ell}^2}, &r_{ij}>r_c.
\end{cases}
\end{equation}
Substituting Eq.~\eqref{eq::3.36} into Eq.~\eqref{eq::3.35} yields 
\begin{equation}\label{eq::final3}
E_{G}^{\text{down}}\simeq\dfrac{r_c}{4\pi }\sum_{\ell=-\infty}^{-1}\left[w_{\ell}e^{-r_c^2/s_{\ell}^2}\sum_{i,j}\frac{q_iq_j}{r_{ij}}e^{-2r_c(r_{ij}-r_c)/s_{\ell}^2}H(r_{ij}-r_{c})\right]
\end{equation}
which is scaled as $O(w_{-1}e^{-r_c^2/s_{-1}^2})$. Since the factor $e^{-2r_c(r_{ij}-r_c)/s_{\ell}^2}$ decays rapidly as $r_{ij}>r_c$ and $H(r_{ij}-r_c)$ has support only when $r_{ij}>r_c$, the main contribution of $E_{G}^{\text{down}}$ comes from neighbours within a small neighborhood outside the cutoff $r_c$ of each particle $i$. This is not surprising because the Gaussian $e^{-r^2/s_{\ell}^2}$ has rapid decay in real space when $\ell\leq-1$ and $r\geq r_c$. 

Combining Eq.~\eqref{eq::Uerr} with the estimation of $E_{T}$, $E_{G}^{\text{up}}$ and $E_{G}^{\text{down}}$ in Section~\ref{subsec::1}-\ref{subsec::3}, we finish the proof of Theorem~\ref{thm::thm4} for $U_{\text{err}}$. The proof for $\bm{F}_{\text{err}}(\bm{r}_i)$ is similar and is given in Appendix~\ref{app::Ferr}.

\section{Extension to \texorpdfstring{$C^1$}~-continuity and further discussions}\label{subsec::extension}
Until now, our main findings and proofs, as discussed in Sections~\ref{subsec::mainresult} and \ref{proof}, have focused solely on $C^0$-continuous u-series. However, since one needs to use force for MD time stepping, it is suggested in~\cite{DEShaw2020JCP} that $C^1$-continuous u-series satisfying Eqs.~\eqref{eq::C0}-\eqref{eq::C1} is practically more advantageous as the force continuity is ensured as well. Theorem \ref{thm::C1continuous} extends the results of $C^0$-continuous u-series, Theorem~\ref{thm::thm4}, to the $C^1$-continuous case. 
\begin{theorem}
\label{thm::C1continuous}
Assume that $s_{M+1}\gg L\geq 2r_c$ and $s_0<r_c$. For the $C^1$-continuous u-series, $U_{\emph{err}}$ and $\bm{F}_{\emph{err}}(\bm{r}_i)$ hold the following estimates,
\begin{equation}\label{eq::errorfinalC1}
	\begin{split}
	|U_{\emph{err}}|&\simeq O\left((\log b)^{-3/2}e^{-\pi^2/2\log b}+b^{-M}+w_{-1}e^{-r_c^2/s_{-1}^2}-(\omega-1)w_{0}e^{-r_c^2/s_0^2}\right),\\[1.em]
	|\bm{F}_{\emph{err}}(\bm{r}_i)|&\simeq O\left((\log b)^{-3/2}e^{-\pi^2/2\log b}+b^{-3M}+\frac{w_{-1}}{(s_{-1})^{2}}e^{-r_c^2/s_{-1}^2}-\frac{(\omega-1)w_{0}}{(s_0)^{2}}e^{-r_c^2/s_0^2}\right),
	\end{split}
	\end{equation}
	where $\simeq$ indicates ``asymptotically equal'' as $b\rightarrow 1$.
\end{theorem}
\begin{proof}
Similar to Eq.~\eqref{eq::Uerr}, the error in electrostatic energy can be divided into three parts: $U_{\text{err}}=E_{T}+E_{G}^{\text{up}}+E_{G}^{\text{down}}$. In the $C^1$-continuous u-series, an extra parameter $\omega$ is introduced to fit the force continuity condition Eq.~\eqref{eq::C1}. Following the proof in Section~\ref{proof}, it is clear that this parameter does not alter the decay order of $E_{T}$ and $E_{G}^{\text{up}}$, but only affects $E_{G}^{\text{down}}$. By replacing Eq.~\eqref{eq::Gdowndefi} with
\begin{equation}
G_\textup{down}(r)= \left[\sum_{\ell=-\infty}^{-1}w_\ell e^{-r^2/s_\ell^2}-(\omega-1)w_{0}e^{-r^2/s_0^2}\right] H(r-r_c) 
\end{equation}
and following the proofs of Proposition~\ref{prop::prop3} and Section~\ref{subsec::3}, one can obtain:
\begin{equation}\label{eq::EgDownC1}
E_{G}^{\text{down}}\simeq \frac{r_c}{4\pi}\left[\sum_{\ell=-\infty}^{-1}w_{\ell}e^{-r_c^2/s_{\ell}^2}g_{\ell}(r_c)-(\omega-1)w_0e^{-r_c^2/s_{0}^2}g_{0}(r_c)\right],
\end{equation}
where
\begin{equation}\label{eq::glrc}
g_{\ell}(r_c):=\sum_{i,j}\frac{q_iq_j}{r_{ij}}e^{-2r_c(r_{ij}-r_c)/s_{\ell}^2}H(r_{ij}-r_c).
\end{equation}
Eq.~\eqref{eq::EgDownC1} indicates that $|E_{G}^{\text{down}}|\simeq O(w_{-1}e^{-r_c^2/s_{-1}^2}-(\omega-1)w_{0}e^{-r_c^2/s_0^2})$. By computing the gradient of $E_{G}^{\text{down}}$ and following the analysis in Eqs.~\eqref{eq::AppC1}-\eqref{eq::force1}, one further has:
\begin{equation}\label{eq::FGDownC1}
|\bm{F}_{G}^{\text{down}}(\bm{r}_i)|\simeq O\left(\frac{w_{-1}}{(s_{-1})^{2}}e^{-r_c^2/s_{-1}^2}-\frac{(\omega-1)w_{0}}{(s_0)^{2}}e^{-r_c^2/s_0^2}\right).
\end{equation}
The decay order of the other two terms, $\bm{F}_{T}$ and $\bm{F}_{G}^{\text{up}}$, remains unaffected by $\omega$. The proof is finished.
\end{proof}
Note that an additional condition $s_0<r_c$, compared to Assumption~\ref{ass:1}, is required in Theorem~\ref{thm::C1continuous} to ensure the validity of the asymptotic expansion Eq.~\eqref{eq::ErfcExp}. Comparing Theorem~\ref{thm::thm4} and Theorem~\ref{thm::C1continuous}, one observes that $C^1$-continuous u-series may yield lower truncation error than $C^0$-continuous u-series as $\omega\rightarrow 1$. Corollary~\ref{corr::C1continuous} presents a rule to set parameters of $C^1$-continuous u-series to guarantee a user-defined tolerance of force in MD. 
\begin{corollary}\label{corr::C1continuous}
Given an error tolerance $\varepsilon>0$, one first sets $b$ and $M$ according to Corollary~\ref{coroll:parameter}, and finds proper $\omega$ and $r_c$ so that
$|\bm{F}_{G}^{\emph{down}}(\bm{r}_i)|\sim\varepsilon/3$ using Eq.~\eqref{eq::FGDownC1}. The error of $C^1$-continuous u-series is thus $\sim\varepsilon$ uniformly over all $r\in\Omega$.
\end{corollary}
Theorem~\ref{thm::C1continuous} can be further extended to $C^p$-continuity with $p\geq 2$. For instance, fine-tuning the width $s_0\rightarrow \gamma s_0$, with $\gamma$ being a parameter, of the narrowest Gaussian achieves the $C^2$-continuity   
\begin{equation}
\frac{2}{r_c^3}-\frac{d^2}{dr^2}\mathcal{F}_{b}^{\sigma}(r){\Big{|}}_{r=r_c}=0
\end{equation}
and so forth. However, it was found~\cite{DEShaw2020JCP} that the actual behavior of $C^2$-continuity may not be better than $C^1$-continuity.  
This agrees with our analysis, which indicates that both conditions $r_c>\gamma s_0$ associated with $\omega\rightarrow 1$ should be met. Unfortunately, the values of $r_c$ required for these conditions are impractically large. This issue may become worse for higher-order continuity with $p>2$. Overall, the $C^1$-continuous u-series behaves well and is already sufficient for purposes relevant to MD simulations.
 
It was numerically observed~\cite{DEShaw2020JCP} that $C^p$-continuous u-series exhibit an energy convergence rate of $O(1/r_c^{p+1})$ for $p\in\{0,1,2\}$. Error estimates in this paper also provide insight into this conjecture. Considering the $C^1$-continuous u-series,  Eqs.~\eqref{eq::EgDownC1} and \eqref{eq::glrc} suggest that the primary contribution arises from $i$-$j$ pairs where $r_{ij}$ approaches $r_c$. Let $r_{>}=\min_{i,j}\{r_{ij}|r_{ij}>r_c\}$. By replacing $H(r-r_c)$ with $H(r-r_>)$ in Eq.~\eqref{eq::Gdowndefi} and following the similar procedure of the proof, one can derive
\begin{equation}
\begin{split}
|E_{G}^{\text{down}}|&\simeq \sum_{\ell=-\infty}^{-1}w_{\ell}e^{-r_>^2/s_{\ell}^2}-(\omega-1)w_0e^{-r_>^2/s_0^2}\\
&=\sum_{\ell=-\infty}^{-1}w_{\ell}e^{-r_>^2/s_{\ell}^2}-\frac{1}{r_>}+\sum_{\ell=0}^{\infty}w_{\ell}e^{-r_>^2/s_{\ell}^2}+O\left(\frac{1}{r_c^2}\right)
\end{split}
\end{equation}
by simple Taylor expansion and the definition of $\omega$, where the first three terms form the error of bilateral series approximation Eq.~\eqref{BSA} and the last equality is due to $C^1$-continuity of the u-series at $r_c$. As $b\rightarrow 1$ and $r_>\rightarrow r_c$, the approximation error vanishes, yielding $E_{G}^{\text{down}}\simeq O(1/r_c^2)$. These arguments can be extended to $C^p$-continuous u-series with $p\neq 1$. However, we remark that the $1/r_c^{p+1}$ decay rate is not always observable. For instance, in the NaCl crystalline lattice discussed in Section~\ref{subsec::madelung}, the lattice size satisfies $r_>=2r_c$. As a consequence, $E_{G}^{\text{down}}$ becomes negligible, hence rendering the $1/r_c^{p+1}$ decay unobservable.

It is worth noting that the SOG approximation is particularly useful in constructing fast algorithms because it allows for the separation of variables. For example, the bilateral series is also utilized to construct separated representations in multiresolution methods~\cite{BEYLKIN2007235,beylkin2012multiresolution}. When employing SOG-based methods on $1/r$, a perplexing issue arises: while the periodic summation of $1/r$ is conditionally convergent under the charge neutrality condition, the summation for finite Gaussians is unequivocally absolutely convergent. The existence of this gap in the u-series is not surprising, as the conditionally convergent component only emerges in the zero-frequency Fourier mode~\cite{hu2014infinite}. If tinfoil boundary conditions are set for $r\rightarrow \infty$, with dielectric permittivity approaching infinity and the zero-frequency mode vanishing, the remainder term becomes convergent and can be accurately represented by the SOG series. In cases of finite permittivity as $r\rightarrow \infty$, the value of zero-frequency mode is  consistent with that of $1/r$ and  depends on the specific summation order \cite{beylkin2012multiresolution}, requiring the addition of the proper correction to the u-series. Further discussions can be found in Section~3.2 of~\cite{Liang2023SISC}.

\section{Closed formulae}\label{sec::closed}
The convergence rate given above is valid for arbitrary systems, whereas the prefactors depend on the particle configuration. Determining these prefactors requires expensive computational cost. To provide a strategy for optimizing parameter setup in simulations, we proceed to derive closed-form formulae of cutoff errors by estimating $\delta U=U_{\text{err}}$ and the standard deviation of force 
\begin{equation}\label{eq::deltaFerr}
\delta \bm{F}_{\text{err}}=\sqrt{\frac{1}{N}\sum_{i=1}^{N} \bm{F}^2_{\text{err}}(\bm{r}_i)}
\end{equation}
which can be described as the average over all particles. In order to achieve this, we use the ideal-gas assumption, i.e., particles with distance greater than the cutoff radius $r_c$ are assumed to be uniformly distributed without correlation. This assumption is also used for error estimate of the Ewald summation \cite{kolafa1992cutoff}. The detailed description of the ideal-gas assumption is provided in Appendix \ref{Sec::Uncorrelation}. Discussions in this
section focuses on the $C^0$-continuous u-series, and the extension to $C^p$-continuous cases is straightforward.

We start with the approximating error $E_T$ of the trapezoidal rule. By Proposition \ref{prop::prop1} and using Eq.~\eqref{eq::K1}, one has,
\begin{equation}\label{eq::ETfinal}
E_T= \frac{1}{2}\sum_{i=1}^{N}\sum_{m\neq 0}\mathcal{C}(m)\sum_j{'} q_i q_j\zeta_{ij}(m)
\end{equation}
where $\zeta_{ij}(m):=r_{ij}^{\alpha_m-1}H(r_{ij}-r_c)$ and 
and the prime represents that the summation runs over all particles in the box and their periodic images at $\bm{r}_j+\bm{n}L$ for $\bm{n}$ being integer vectors. By the ideal-gas assumption in Appendix \ref{Sec::Uncorrelation}, this part of energy error can be approximated by
\begin{equation}\label{4.3eq::Et2}
\delta E_T^2\approx \left[\frac{1}{2}\sum_{m\neq 0}\mathcal{C}(m)\right]^2\dfrac{Q^2}{V} \int_{r_c}^{\infty}4\pi r^{2\alpha_m} dr,
\end{equation}
with $Q=\sum_{i=1}^{N}q_i^2$. Here $\alpha_m$ is a pure imaginary number, and the integral in Eq. \eqref{4.3eq::Et2} does not converge. Physically, this is due to that the long-range part is included in the kernel $T(\cdot)$. 
However, one can treat the integral $\int_{r_c}^{\infty} r^{2\alpha_m} dr$ as the half-Mellin transform such that the integral can be mathematically interpreted in the sense of meromorphic continuations  \cite{zagier2006mellin}
\begin{equation}\label{eq::GM}
\int_{r_c}^{\infty}4\pi r^{2\alpha_m} dr=-\dfrac{4\pi}{1+2\alpha_m}r_c^{1+2\alpha_m}.
\end{equation}
Actually, by integral by parts, the improper integral can be written as the meromorphic extension Eq.~\eqref{eq::GM} plus some infinite term, and this infinite term can be simply ignored, in coincident with the tinfoil boundary condition assumed in the u-series method \cite{DEShaw2020JCP}.  

Substituting Eq.~\eqref{eq::GM} into Eq.~\eqref{4.3eq::Et2} and taking the square root, one obtains 
\begin{equation}
\delta E_T\approx Q\sum_{m\neq 0}\mathcal{C}(m)\sqrt{-\frac{\pi r_c^{2\alpha_m}}{(1+2\alpha_m)V}}.
\end{equation}
Since $\mathcal{C}(m)$ decays rapidly with the increase of $|m|$, one can safely truncate $m$ at $\pm 1$ to estimate $\delta E_T$.  Similarly, for a closed formula for $\delta \bm{F}_{T}$, one begins with
\begin{equation}
\bm{F}_{T}(\bm{r}_i)=-\sum_{m\neq 0}\mathcal{C}(m)\left(\alpha_m-1\right) \sum_j{'} q_i q_j\bm{\zeta}_{ij}(m)
\end{equation}
with $\bm{\zeta}_{ij}(m)=r_{ij}^{\alpha_m-3}\bm{r}_{ij}H(r_{ij}-r_c).$
Again, the ideal-gas assumption leads us to
\begin{equation}
\label{Eq::Ft}
\delta \bm{F}_{T}\approx Q\sum_{m\neq 0}\sqrt{\frac{4\pi r_c^{2\alpha_m-1}}{\left(1-2\alpha_m\right)NV}}\left(\alpha_m-1\right)\mathcal{C}(m).
\end{equation}
Note that the integral which leads to Eq.~\eqref{Eq::Ft} is convergent, different from the estimate of $\delta E_T$. One can also  truncate $m$ at $|m|=1$ to estimate the force error.

We now consider the term of $\delta E_{G}^{\text{down}}$ or $\delta \bm{F}_{G,i}^{\text{down}}$. One has
\begin{equation}
\delta E_{G}^{\text{down}}=\frac{1}{2}\sum_{i=1}^N \sum_j{'} q_i q_j\zeta_{ij} 
\end{equation}
with $\zeta_{ij}=H\left(r_{ij}-r_c\right)\sum_{\ell=-\infty}^{-1} w_{\ell} e^{-r_{ij}^2 / s_{\ell}^2}.$
By the ideal-gas assumption (Appendix \ref{Sec::Uncorrelation}), one obtains
\begin{equation} \label{eq::EGdowndeelta}
	\delta E_{G}^{\text{down}}\approx \frac{1}{2} \sum_{\ell=-\infty}^{-1} w_{\ell} \sqrt{\frac{Q^2}{V} \int_{r_c}^{\infty}e^{-2 r^2 / s_{\ell}^2} 4 \pi r^2 d r}.
\end{equation}
The integral in Eq.~\eqref{eq::EGdowndeelta} can be asymptotically approximated by the Laplace method~\cite{murray2012asymptotic} considering that $2r_c/s_{-1}$ is big under Assumption \ref{ass:1}. This allows us to approximate $\delta E_{G}^{\text{down}}$ by
\begin{equation}
\delta E_{G}^{\text{down}} \approx \frac{1}{2}Q\sqrt{\frac{\pi  r_c}{ V}} \sum_{\ell=-\infty}^{-1} w_{\ell} s_{\ell} e^{-r_c^2 / s_{\ell}^2}.
\end{equation}
Similarly, one has the following estimate for the force 
\begin{equation}
\delta \bm{F}_{G}^{\text{down}}\approx Q\sqrt{\frac{\pi  r_c}{ NV}}\sum_{\ell=-\infty}^{-1}\sqrt{\frac{ 4r_c^2+3s_{\ell}^2}{s_{\ell}^2/w_{\ell}^2}}e^{-r_c^2 / s_{\ell}^2}.
\end{equation}

For $E_{G}^{\text{up}}$ and $\bm{F}_{G}^{\text{up}}(\bm{r}_i)$, one should estimate them in the Fourier space since the corresponding summations converge slowly in real space. The analysis in Section \ref{subsec::2} and Appendix~\ref{app::Ferr} shows that
\begin{equation}\label{eq::EGup}
E_{G}^{\text{up}}\approx \frac{\log b}{\sqrt{\pi}}\sum_{\ell=M+1}^{\infty}\sum_{i,j=1}^{N}\frac{q_iq_j}{s_{\ell}}e^{-r_{ij}^2/s_{\ell}^2}+\frac{\sqrt{\pi}\log b}{(b-1)s_{M}}\sum_{i,j=1}^{N}q_iq_jH(r_{ij}-r_{c}),
\end{equation} 
and
\begin{equation}\label{eq::FGupr}
\bm{F}_{G}^{\text{up}}(\bm{r}_i)\approx \frac{2\log b}{\sqrt{\pi}}q_i\sum_{\ell=M+1}^{\infty}\sum_{j=1}^{N}\frac{q_j\bm{r}_{ij}}{s_{\ell}^3}e^{-r_{ij}^2/s_{\ell}^2}.
\end{equation}
Eqs.~\eqref{eq::EGup} and \eqref{eq::FGupr} are accurate, but it is difficult to use in practical simulations due to the expensive 
cost to compute the summation. Since $r_c\ll s_{M+1}$, by the second order Taylor expansion  
\begin{equation}
\sum_{\ell=M+1}^{\infty}\sum_{i,j=1}^{N}\frac{q_iq_j}{s_{\ell}}e^{-r_{ij}^2/s_{\ell}^2}= O(1/s_{M}^3)
\end{equation}
and the charge neutrality condition, one gets
\begin{equation}\label{eq::Egup1.1}
\delta E_{G}^{\text{up}}\approx\frac{\sqrt{\pi}\log b}{(b-1)s_{M}}\sum_{i,j=1}^{N}q_iq_jH(r_{ij}-r_{c}).
\end{equation}
With Eq.~\eqref{eq::Egup1.1}, $\delta E_{G}^{\text{up}}$  can be easily computed with $O(N)$ operations. 
For $\bm{F}_{G}^{\text{up}}$, one uses the ideal-gas assumption to obtain
\begin{equation}
\delta\bm{F}_{G}^{\text{up}}\approx \dfrac{2Q\log b }{\sqrt{\pi NV}s_{M+1}^3}\mathcal{B}(L/2,s_{M+1}),
\end{equation}
where
\begin{equation}
\mathcal{B}(x,y)=\left[\frac{3\sqrt{2\pi^3}}{16}y^5\erf\left(\sqrt{2}x/y\right)-\left(\frac{3}{4}xy^4\pi+x^3y^2\pi\right)e^{-2x^2/y^2}\right]^{1/2}.
\end{equation}

To summarize the above analysis, the errors in energy and force are approximated by
\begin{equation}\label{eq::deltaUerr}
\begin{split}
\delta U_{\text{err}}\approx& Q\left[\sum_{m\neq 0}\mathcal{C}(m)\sqrt{-\frac{\pi r_c^{1+2\alpha_m}}{(1+2\alpha_m)V}}+\sqrt{\frac{\pi r_c}{4 V}} \sum_{\ell=-\infty}^{-1} w_{\ell} s_{\ell} e^{-r_c^2 / s_{\ell}^2}\right]\\
&+\frac{\sqrt{\pi}\log b}{(b-1)s_{M}}\sum_{i,j=1}^{N}G_{ij},
\end{split}
\end{equation}
\begin{equation}\label{eq::deltaFerrfinal}
\resizebox{.95\hsize}{!}{$
\begin{aligned}
\delta F_{\text{err}}
\approx&\frac{Q}{\sqrt{NV}}\left[\sum_{m\neq 0}\sqrt{\frac{4\pi r_c^{2\alpha_m-1}}{\left(1-2\alpha_m\right)}}\left(\alpha_m-1\right)\mathcal{C}(m)+\sum_{\ell=-\infty}^{-1}\sqrt{\frac{\pi (4r_c^3+3r_cs_{\ell}^2)}{s_{\ell}^2/w_{\ell}^2}}e^{-r_c^2 / s_{\ell}^2}\right]\\
&+\dfrac{2Q\log b }{\sqrt{\pi NV }s_{M+1}^3}\mathcal{B}(L/2,s_{M+1}).
\end{aligned}$}
\end{equation}
Note that the first components of both $\delta U_{\text{err}}$ and $\delta F_{\text{err}}$ are independent of the number of truncated Gaussians, $M$, and thus control the final accuracy of the energy and the force, respectively. Accordingly, the precision that grows with $M$ is determined by their second components, decaying with $O(1/s_{M})$ and $O(1/s_{M}^3)$, respectively.
For the parameter optimization, we remark that if the cutoff radius $r_c$ is specified before the simulation, one can use the first components to determine $b$ and  $\sigma$, achieving a priori error level $\varepsilon$. Typically, $\varepsilon=10^{-3}$ and $10^{-4}$ are two widely used choice. Then, one uses the second components to determine the number of truncated Gaussians, $M$, such that the cutoff error is at the level of $\varepsilon$.

\section{Numerical results}
\label{example}
We present numerical results of two benchmark examples to verify our error estimates. All the calculations and simulations were conducted by the self-coded u-series method implemented in a modified version of LAMMPS \cite{thompson2022lammps} (version 23June2022). 

\subsection{Madelung constant of crystalline lattice}\label{subsec::madelung}
The first example is the calculation of the Madelung constant for NaCl
cubic-like crystalline lattice \cite{madelung1918elektrische}. The Madelung constant is used to determine the electrostatic energy of an ion
in a crystal. In the setup, $N=8$ unit source charges are evenly arranged on a lattice grid of edge length $L=40\mathring{A}$, with neighboring
charges having opposite charge amount. The relative error is computed by comparing the numerical results with the exact value of the Madelung constant $-1.74756459463318219$.
The numerical error is estimated by calculating separately the three components of $
U_{\text{err}}=E_{T}+E_{G}^{\text{up}}+E_{G}^{\text{down}}$. 
Since the number of particles $N$ is relatively small, one can compute $E_T$ directly by using Eq.~\eqref{Uerr1} where the integral is evaluated using the adaptive Gauss-Legendre quadrature \cite{shampine2008vectorized}. The other two terms, $E_{G}^{\text{up}}$ and $E_{G}^{\text{down}}$, are estimated using the last formula in Eqs.\eqref{eq::final2} and \eqref{eq::final3}, respectively. Since the minimum distance between two charges is much larger than $r_c$, one has
\begin{equation}\label{eq::EGup1}
E_{G}^{\text{up}}\approx \frac{\sqrt{\pi}Q\log b}{(b-1)s_{M}}
\end{equation}
in this case and $E_{G}^{\text{down}}$ is small enough and can be neglected. It is noted that $E_{G}^{\text{up}}$ given in Eq.~\eqref{eq::EGup1} does not work for small $M$ because of $s_{M}\gg r_c$ in Assumption \ref{ass:1}. As $M\rightarrow 0$, the neighbors within cutoff $r_c$ of each charge is only the charge itself and one has $U_{\text{err}}\rightarrow U$. The error in Madelung constant can be accurately estimated by $\min\{|U_{\text{err}}|,U\}/N$.

Fig.~\ref{fig:L} shows the error convergence with the number of truncated Gaussians, $M$, where $r_c=10\mathring{A}$ is taken and the results
of five different $b$ are displayed. It is
observed that the  convergence rate agrees with the theoretical estimate. To better display the legend, the values of $b$ in the legend of Fig.~\ref{fig:L} will be rounded to up to three decimal places, and the detailed parameter sets are shown in Table \ref{tabl:parameter}.
\\

\begin{figure*}[ht]	
	\centering
	\includegraphics[width=0.7\textwidth]{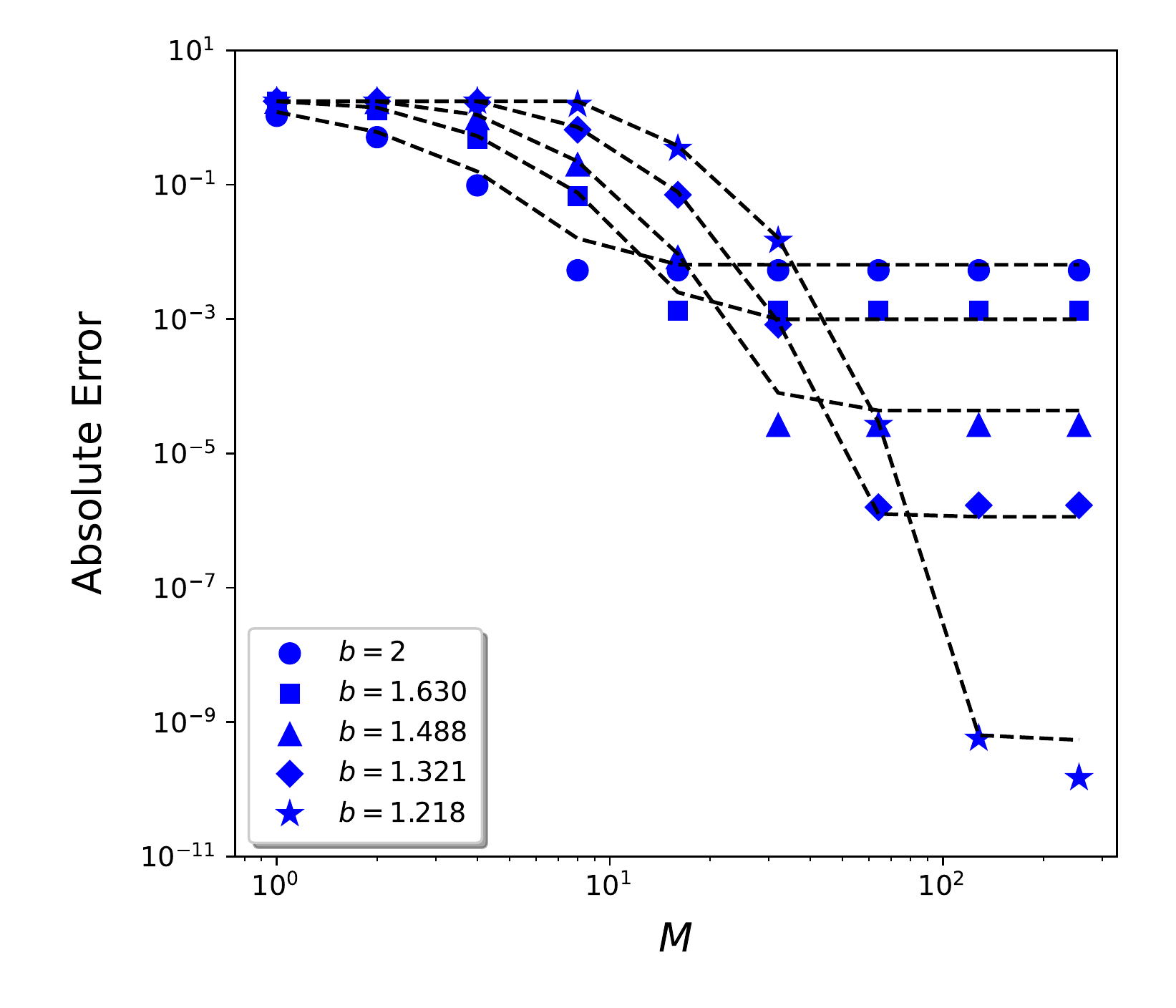}
	\caption{Absolute error as a function of the number of truncated terms, $M$, for different $b$, compared to the exact Madelung constant. The symbols indicate the corresponding errors by the u-series method, and the lines display theoretical estimates in this work.}
	\label{fig:L}
\end{figure*}

\subsection{All-atom water system}
We perform calculations on an all-atom bulk water system using the extended simple point charge (SPC/E) \cite{berendsen1987missing} force field to examine the accuracy of our error estimate, comparing with the reference solutions which are produced by the u-series method with machine precision for given configurations. The system includes $7208$ water molecules confined in a cubic box of initial side length $60\mathring{A}$. We measure the relative error on the energy and force by taking the average of results from $1000$ configurations uniformly sampled from a $100$~$ps$ MD simulation.

\begin{figure*}[!ht]	
	\centering
	\includegraphics[width=0.97\textwidth]{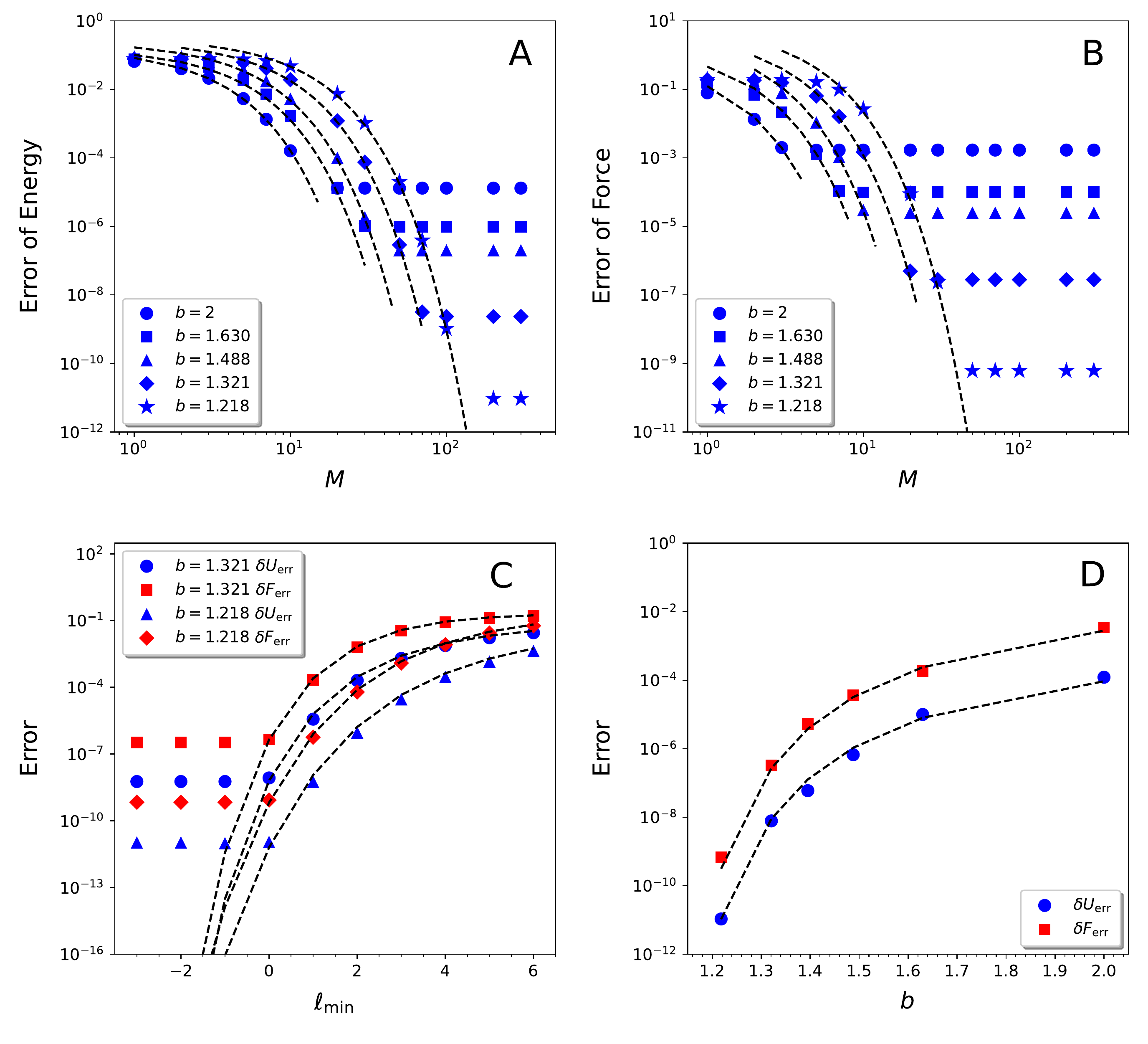}
	\caption{Errors in calculations of an SPC/E pure water system using $C^0$-continuous u-series. Panel (A-B) display  the relative errors of the energy and the force, respectively, as function of $M$. The dashed lines in both panels indicate $O(b^{-M})$ and $O(b^{-3M})$ scaling, respectively. Panel (C) plots the relative error regarding the energy and the force against the minimal index of Gaussians for $b=1.321$ and 1.218. Panel (D) shows the relative errors as function of $b$. The dash lines in (C-D) show the theoretical estimate using Eqs.~\eqref{eq::final3} and \eqref{eq::Uerr1} for energy and  Eqs.~\eqref{eq::force1} and \eqref{app::eqerr1} for force.
	}
	\label{fig:Wa1}
\end{figure*}

We first attempt to demonstrate that the convergence rate given in Theorem \ref{thm::thm4} is accurate. Fig.~\ref{fig:Wa1}(A-B) show the error as a function of the number of truncated terms $M$ for different $b$, where $r_c$ takes $10\mathring{A}$. The dashed lines in (A) and (B) show the theoretical $O(b^{-M})$ and $O(b^{-3M})$ convergence rates, respectively. All the numerical results show well agreement with our theoretical predictions for the error of both the energy and force. These results demonstrate that our estimate of the convergence order of $E_{G}^{\text{up}}$ is accurate. Table \ref{tabl:parameter} shows the detailed parameter sets for SOG approximations. When $M$ is large enough, the error will stop decaying and the convergence accuracy should depend on both $E_{T}$ and $E_{G}^{\text{down}}$.

We examine the existence of $E_{G}^{\text{down}}$. For this purpose, we fix $r_c=10~\mathring{A}$ and $M=200$ in producing Fig.~\ref{fig:Wa1}(C) such that the range of $E_{G}^{\text{up}}$ will not become the leading term in the error. We then tune the minimum index of Gaussians, denoted as $\ell_{\min}$, in the SOG decomposition from $-3$ to $6$. The dashed lines in Fig.~\ref{fig:Wa1}(C) show the theoretical
$O\left(w_{\ell_{\text{min}-1}}\exp({-r_c^2/s_{\ell_{\text{min}}-1}^2})\right)$and  $O\left(w_{\ell_{\text{min}-1}}s_{\ell_{\text{min}}-1}^{-2}\exp(-r_c^2/s_{\ell_{\text{min}}-1}^2)\right)$
scalings, respectively. It is clearly observed that all the numerical results produced by the u-series method show the corresponding scaling with respect to $\ell_{\text{min}}$, in agreement with our estimate of $E_{G}^{\text{down}}$. It is noted that the error stops decaying when $\ell_{\text{min}}$ is less than $0$, because $E_{T}$ will become dominant at this case.
One can also verify the existence of $E_{T}$ in the cutoff error of the u-series method in a similar manner. Here, we fix $r_c=10~\mathring{A}$, $\ell_{\text{min}}=0$, and $M=200$ such that the range of $E_{G}^{\text{up}}$ and $E_{G}^{\text{down}}$ will not affect the error significantly. The numerical results  are displayed in Fig.~\ref{fig:Wa1}(D), where the dashed lines show theoretical  $O\left(e^{-\pi^2/2\log b}(\log b)^{-3/2}\right)$ scalings for both the energy and force. One can clearly observe the correctness of our estimate on the convergence order 
of $E_{T}$. 

\renewcommand\arraystretch{1.4}
\begin{table}[!htbp]
	\caption{Typical parameters and corresponding errors for the $C^0$ u-series method with different $b$ and $\sigma$ at ${r_c}=10\mathring{A}$. $M$ is the minimum number of Gaussian terms required to achieve the error.}
	\centering
 	\scalebox{0.9}{
	\begin{tabular}{c|c|cc|cc}
		\hline \multirow{2}{*}{$b$}& \multirow{2}{*}{$\sigma$} & \multicolumn{2}{c|}{Energy} & \multicolumn{2}{c}{Force} \\
		 \cline{3-6} &  &  Error & $M$ &  Error& $M$ \\
		\hline 2 & $5.027010924194599$ & $1.22\hbox{E}-4$ & $12$ & $3.48\hbox{E}-3$ & $4$ \\
		\hline $1.62976708826776469$ & $3.633717409009413$ & $1.06\hbox{E}-5$ & $25$ & $1.83\hbox{E}-4$ & $7$ \\
		\hline $1.48783512395703226$ & $2.662784519725113$ & $6.69\hbox{E}-7$ & $34$ & $3.68\hbox{E}-5$ & $10$ \\
		\hline $1.39514986274321621$ & $2.577606398612261$ & $5.96\hbox{E}-8$ & $50$ & $5.22\hbox{E}-6$ & $14$\\
		\hline $1.32070036405934420$ & $2.277149356440992$ & $7.80\hbox{E}-9$ & $64$ & $3.27\hbox{E}-7$ & $20$ \\
		\hline $1.21812525709410644$ & $1.774456369233284$ & $1.07\hbox{E}-11$ & $124$ & $6.79\hbox{E}-10$ & $40$ \\
		\hline
	\end{tabular}}
	\label{tabl:parameter}
\end{table}

The results of Fig.~\ref{fig:Wa1} demonstrate the convergence order given in Theorem \ref{thm::thm4}. It is also crucial to compare our closed-form theoretical estimate with the practical results. Different from the Madelung constant calculation, Eqs.~\eqref{Uerr1}, \eqref{eq::final2} and \eqref{eq::final3} cannot be used due to the expensive cost. Instead, we use the closed formulae Eqs.~\eqref{eq::deltaUerr} and \eqref{eq::deltaFerrfinal} to compare with the experimental results. It is observed in Fig.~\ref{fig:Wa5} that the theoretical estimates work well for both the energy and force. This agreement is convincing because the all-atom SPC/E water is a rather typical system and both the experimental and estimated values agree in most of $M$. This wide applicability is much required in the process of determining the optimal parameters to reduce the computational costs. We remark that the small difference between the theoretical estimate and the experimental values in terms of the final convergence accuracy is due to the ideal-gas assumption we made in Section \ref{sec::closed}. Since the hydrogen and oxygen within the same water molecule are chemically bonded together, this assumption leads to error. The numerical results show that the theoretical estimate remain valid in real simulations for this non-simple particle system. 

\begin{figure*}[!ht]	
	\centering
	\includegraphics[width=0.97\textwidth]{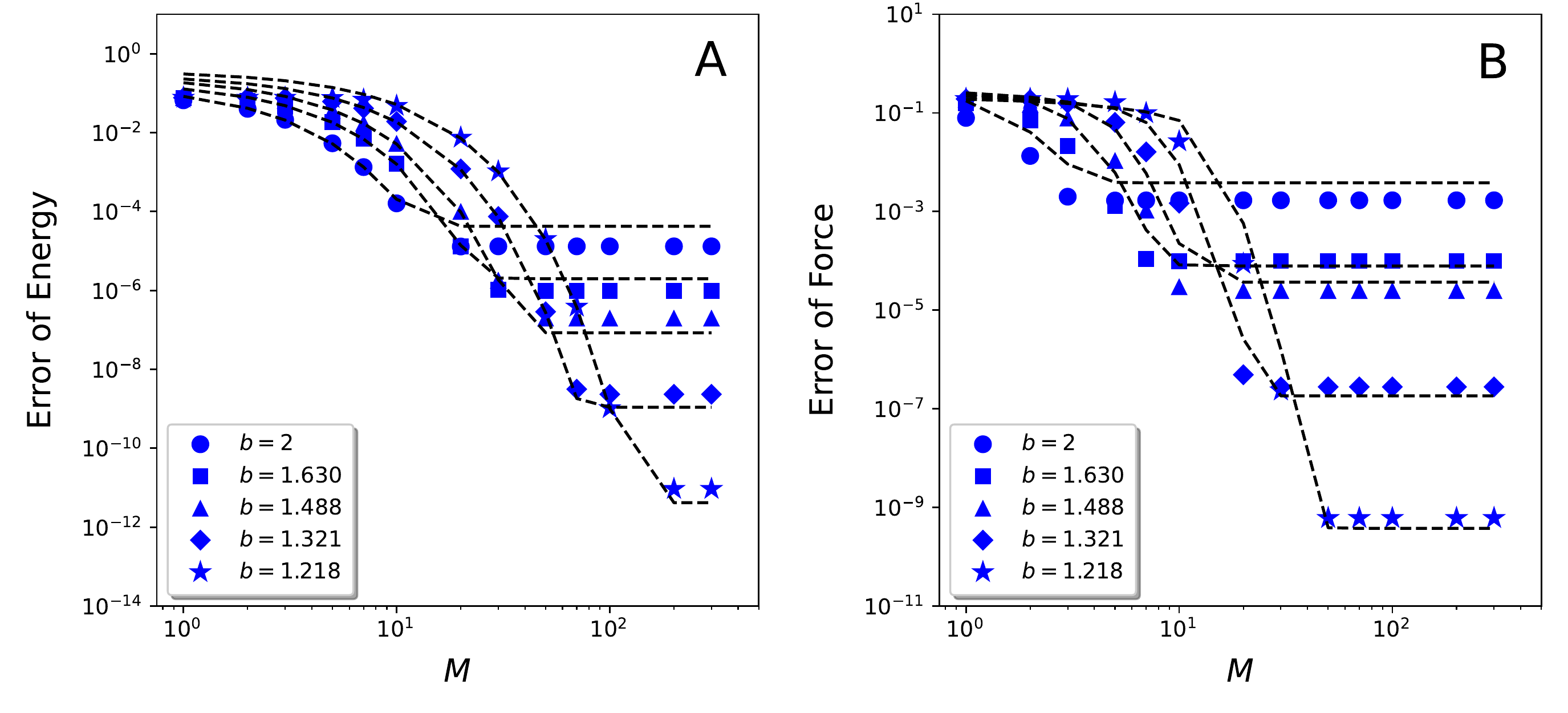}
	\caption{The relative errors of energy and force as a function of the number of truncated terms, $M$, for different $b$ in series of simulations on a SPC/E pure water system. The dashed lines in (A) and (B) plot closed-form theoretical estimates of the energy and force using Eqs.~\eqref{eq::deltaUerr} and \eqref{eq::deltaFerrfinal}.}
	\label{fig:Wa5}
\end{figure*}

Next, we illustrate the convergence rate of $C^1$-continuous u-series given in Theorem~\ref{thm::C1continuous} by numerical results. To explore the existence of $E_G^{\text{down}}$, we set $r_c=10\mathring{A}$ and $M=300$, ensuring that the range of $E_{T}$ and $E_{G}^{\text{up}}$ does not dominate the error. The results are given in Fig.~\ref{fig:Wa6}. Similar to the $C^0$ case, we vary $\ell_{\text{min}}$ from $-3$ to $6$. The dashed lines in Fig.~\ref{fig:Wa6}(A) represent the theoretical scalings for the energy and force. Notably, numerical results agree well with our estimate for $E_{G}^{\text{down}}$ and $\bm{F}_{G}^{\text{down}}$. The limit accuracy of energy and force with increasing $b$ is depicted in Fig.~\ref{fig:Wa6}(B), with detailed parameter sets provided in Table~\ref{tabl:parameterC1}. We observe $O(e^{-\pi^2/2\log b}(\log b)^{-3/2})$ scalings for both energy and force as expected. Comparing Table~\ref{tabl:parameter} with Table~\ref{tabl:parameterC1}, it is evident that the $C^1$-continuous u-series achieves higher accuracy than the $C^0$-continuous case with a similar number of Gaussians. Due to the ensured continuity of forces, this improvement in accuracy is more pronounced in force computation, reaching up to an order of magnitude. These findings suggest that the $C^1$-continuous u-series is more practically useful than the $C^0$-continuous u-series.

\begin{figure*}[t]	
	\centering
	\includegraphics[width=0.97\textwidth]{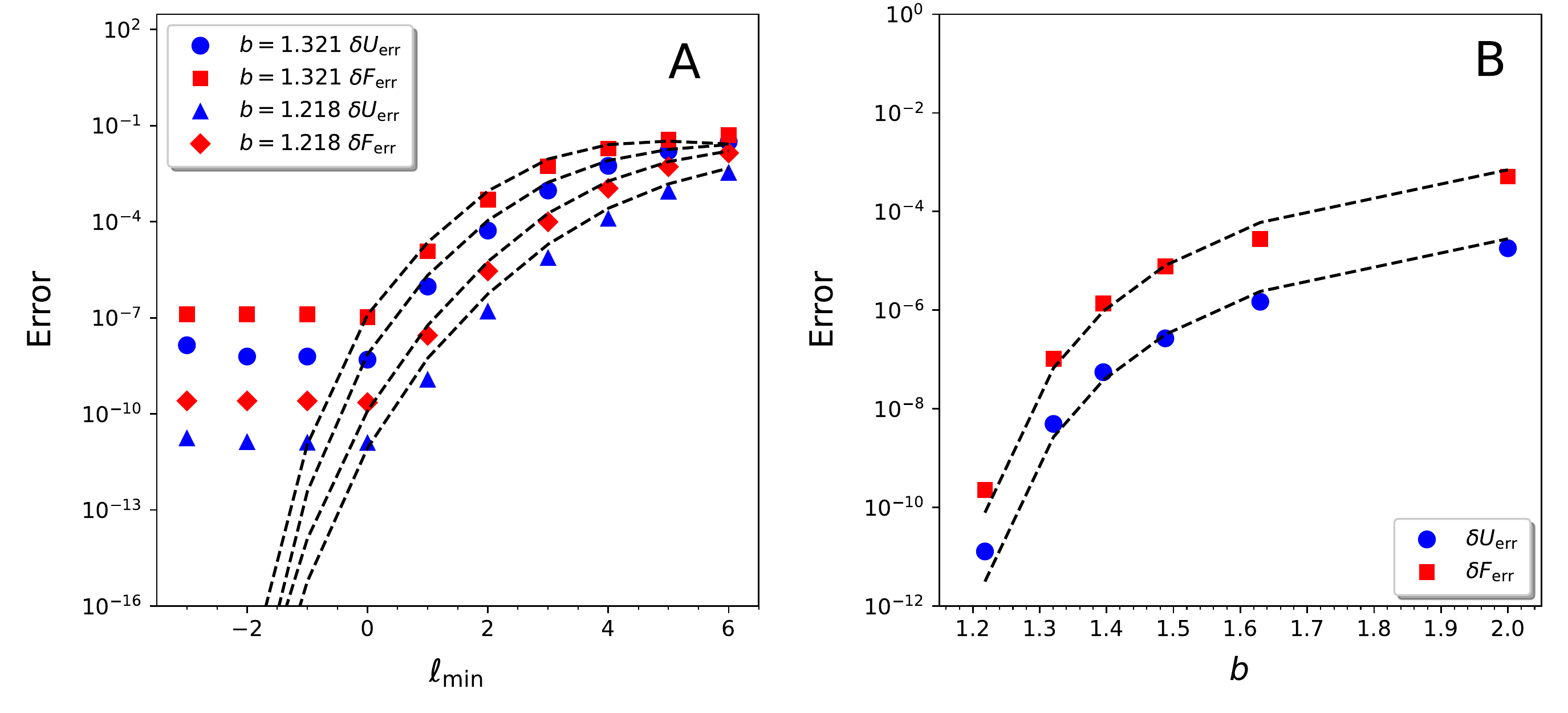}
	\caption{The relative errors of energy and force in an SPC/E pure water system using $C^1$-continuous u-series. Panel (A) reveals the relative error regarding the energy and the force against the minimal index of Gaussians for $b=1.321$ and $b=1.218$. Panel (B) shows the relative errors as function of $b$. The dash lines in (C-D) show the theoretical estimate using Eqs.~\eqref{eq::EgDownC1} and \eqref{eq::Uerr1} for energy and  Eqs.~\eqref{eq::FGDownC1} and \eqref{app::eqerr1} for force.}
	\label{fig:Wa6}
\end{figure*}

\renewcommand\arraystretch{1.4}
\begin{table}[!htbp]
	\caption{Typical parameters and corresponding errors for the $C^1$ u-series method with different $b$ and $\sigma$ at ${r_c}=10\mathring{A}$. $M$ is the minimum number of Gaussian terms required to achieve the error.}
	\centering
	\scalebox{0.9}{
 \begin{tabular}{c|c|c|cc|cc}
		\hline \multirow{2}{*}{$b$}& \multirow{2}{*}{$\sigma$} & 
        \multirow{2}{*}{$\omega$} &
  \multicolumn{2}{c|}{Energy} & \multicolumn{2}{c}{Force} \\
		 \cline{4-7} & & &  Error & $M$ &  Error& $M$ \\
		\hline 2 & $5.027010924194599$ & $0.994446492762232$&$1.80\hbox{E}-5$ & $14$ & $5.13\hbox{E}-4$ & $4$ \\
		\hline $1.62976708826776469$ & $3.633717409009413$ &$1.007806979343806$ & $1.48\hbox{E}-6$ & $26$ & $2.75\hbox{E}-5$ & $7$   \\
		\hline $1.48783512395703226$ & $2.662784519725113$ &$0.991911705759818$ & $2.68\hbox{E}-7$ & $37$ & $7.80\hbox{E}-6$ & $11$   \\
            \hline $1.39514986274321621$ & $2.577606396703941$ &$1.000139836531469$ &  $5.52 \hbox{E}-8$ & $48$ & $1.38\hbox{E}-6$ & $14$   \\
		\hline $1.32070036405934420$ & $2.277149356440992$ &$1.001889141148119$& $4.93\hbox{E}-9$ & $68$ & $1.04\hbox{E}-7$ & $21$   \\
		\hline $1.21812525709410644$ & $1.774456369233284$ &$1.000901461560333$ &  $1.28\hbox{E}-11$ & $126$ & $2.25\hbox{E}-10$ & $41$   \\
		\hline
	\end{tabular}
 }
	\label{tabl:parameterC1}
\end{table}

\section{Conclusion}
\label{conclusion}
In this paper, we present error estimate of the energy and force due to the cutoff of short-range interactions in the u-series method for molecular dynamics, filling in the gap in the theoretical demonstration of the convergence for the method. We also prove the convergence rate with respect to the number of truncated Gaussians, and provide closed-form formulae as approximation of the errors. These results are useful for the parameter setup in simulations using the u-series method, allowing us to find a tradeoff between efficiency and errors in calculations. The accuracy and usefulness of our estimates are examined by conducting  simulations of a NaCl cubic-like crystalline and an SPC/E bulk water system. Extending these theoretical results from fully periodic to partially periodic boundary conditions follows a similar process and will be presented in our subsequent work~\cite{gao2024fast}. Furthermore, we anticipate that the main theorems presented in this paper will inspire further advancements in the convergence analysis of other recently proposed SOG-based fast algorithms~\cite{greengard2023dual} under periodic boundary conditions.

\section*{Acknowledgements}
The authors acknowledge the nancial support from the National Natural Science Foundation of China (grants No.~12071288 and 12325113), Science and Technology Commission of Shanghai Municipality (grant No.~21JC1403700), and the support from the HPC center of Shanghai Jiao Tong University.

\appendix
\section{Proofs of some useful results}
\subsection{Proof of Proposition \ref{prop::prop1}}\label{pf::prop1}
We first need the following lemma.
\begin{lemma}
	\label{lemma::poisson}
	The approximation error between the Coulomb kernel and its bilateral series expansion can be written as the sum of Fourier basis,
	\begin{equation}\label{eq::lemma2}
	\frac{1}{r}-\sum_{\ell=-\infty}^{\infty}w_\ell e^{-r^2/s_\ell^2}=\sum_{m\neq 0}\mathcal{C}(m)r^{\alpha_m-1},
	\end{equation}
	where the definition of $\mathcal{C}(m)$ is given by Eq.~\eqref{eq::Cndef}.
\end{lemma}
\begin{proof}
	Assume $f(r)\in \mathcal{S}(\mathbb{R})$ with $\mathcal{S}$ the space of Schwartz such that the Fourier transform of $f$ is well defined. The Poisson summation formula \cite{stein2011fourier}
	\begin{equation}\label{eq::Poisson}
	\sum_{m\in\mathbb{Z}}\widetilde{f}\left(\frac{m}{h}\right)e^{2\pi it_0\frac{m}{h}}=h\sum_{m\in\mathbb{Z}}f(t_0+mh)
	\end{equation}
	holds for any initial point $t_0\in \mathbb{R}$ and sample step $h>0$. Recall the integral representation of the Coulomb kernel, $1/r=\int_{\mathbb{R}}f(x)dx$, with
	\begin{equation}\label{eq::A.3}
	f(x)=\frac{1}{\sqrt{\pi}}e^{-r^2e^x+\frac{1}{2}x}.
	\end{equation}
By substituting $t_0=-\log 2\sigma^2$, $h=2\log b$ and Eq.~\eqref{eq::A.3} into Eq.~\eqref{eq::Poisson}, the right-hand side of Eq.~\eqref{eq::Poisson} is nothing but the bilateral series expansion of $1/r$, thus one obtains
	\begin{equation}
	\frac{1}{r}-\sum_{\ell=-\infty}^{\infty}w_\ell e^{-r^2/s_\ell^2}=\widetilde{f}(0)-\sum_{m\in\mathbb{Z}}\widetilde{f}\left(\frac{m}{h}\right)e^{2\pi it_0\frac{m}{h}}		
	\end{equation}
	where the fact $1/r=\widetilde{f}(0)$ is used and the Fourier transform of $f$ is given by
	\begin{equation}
	\widetilde{f}(k)=\frac{1}{\sqrt{\pi}}\int_{\mathbb{R}}e^{-r^2e^x+\frac{1}{2}x}e^{-2\pi ik x}dx=\frac{\Gamma(1/2-2\pi ik)}{\sqrt{\pi}}r^{4\pi ik-1}.
	\end{equation}
	The proof is completed.
\end{proof}
Substituting Lemma~\ref{lemma::poisson} into Eq.~\eqref{eq::3.91} yields the Fourier transform of $T(r)$, that is,
\begin{equation}\label{eq::K1k}
\widetilde{T}(k)=\sum_{m\neq 0}\mathcal{C}(m)\widetilde{\phi}_m(k)
\end{equation}
where $\widetilde{\phi}_m(k)$ is the Fourier transform of
\begin{equation}
\phi_m(r)=r^{\alpha_m-1}H(r-r_c).
\end{equation}
To calculate $\widetilde{\phi}_m(k)$, according to Lemma~\ref{lemma1}, the Fourier transform of $\phi_m(r)$ is given by
\begin{equation}
\widetilde{\phi}_m(k)=\dfrac{4\pi}{k}\int_{r_c}^{\infty}\sin(kr)r^{\alpha_m}dr.
\end{equation}
One can use complex-exponential functions to rewrite it as
\begin{equation}
\widetilde{\phi}_m(k)=\dfrac{2\pi}{ki}\int_{r_c}^{\infty}\left(e^{ikr}-e^{-ikr}\right) r^{\alpha_m}dr.
\end{equation}
By the change of variables, one has
\begin{equation}
\widetilde{\phi}_m(k)=\dfrac{2\pi}{k^2}\left[\int_{-ikr_c}^{\infty}e^{-u}\left(\dfrac{u}{-ik}\right)^{\alpha_m}du+\int_{ikr_c}^{\infty}e^{-u}\left(\dfrac{u}{ik}\right)^{\alpha_m}du\right]
\end{equation}
which is the form of Gamma-type function. By the definitions of the Gamma function and the incomplete Gamma function  
\cite{cuyt2008handbook},
\begin{equation}\label{eq::gamma2}
\Gamma(x):=\int_{0}^{\infty}e^{-t}t^{x-1}dt,\quad\text{and}\quad \gamma(a,x):=\int_{0}^{x}t^{a-1}e^{-t}dt,
\end{equation}
one can rewrite $\widetilde{\phi}_n(k)$ as
\begin{equation}\label{eq::derivation}
\resizebox{1\hsize}{!}{$
	\widetilde{\phi}_m(k)=\dfrac{2\pi}{k^2}\dfrac{1}{(-ki)^{\alpha_m}}\left[\Gamma\left(1+\alpha_m\right)-\gamma\left(1+\alpha_m,-kr_c i\right)\right]+\dfrac{2\pi}{k^2} \dfrac{1}{(ki)^{\alpha_m}}\left[\Gamma\left(1+\alpha_m\right)-\gamma\left(1+\alpha_m,kr_c i\right)\right].
	$}
\end{equation}
To proceed, let us introduce a helpful representation of incomplete Gamma function~\cite{cuyt2008handbook}. 
\begin{lemma}\label{lemma::inGamma}
Let $a$ and $x$ be two complex numbers and $\Real(a)>0$. The incomplete Gamma function has an integral representation,
\begin{equation}\label{eq::gammax}
\gamma(a,x)=x^a\int_{0}^{\infty} e^{-at-xe^{-t}}dt.
\end{equation}
\end{lemma}
By taking $a=1+2\alpha_m$ and $b=\pm kr_ci$ in Lemma~\ref{lemma::inGamma} and the change of variable $x=kr_ce^{-u}$, one obtains 
\begin{equation}\label{eq::i0}
\widetilde{\phi}_m(k)=\dfrac{4\pi}{k^{2+\alpha_m}}\left[\cosh\left(\frac{m\pi^2}{\log b}\right)\Gamma\left(1+\alpha_m\right)-\int_{0}^{kr_c}\sin(x)x^{\alpha_m}dx\right].
\end{equation}
Combining Eqs.~\eqref{eq::i0} with \eqref{eq::K1k}, we finish the proof of Proposition \ref{prop::prop1}.

\subsection{Proof of Proposition \ref{prop::prop2}}\label{pf::prop2}
\par For $\widetilde{G}_{\text{up}}$, one can write the Fourier transform as
\begin{equation}\label{eq::UerrUp}
\begin{aligned}
\widetilde{G}_{\text{up}}(k)=&\dfrac{4\pi}{k}\left(\int_{0}^{\infty}-\int_{0}^{r_c}\right)\sin(kr)r\left(\sum_{\ell=M+1}^{\infty}w_\ell e^{-r^2/s_\ell^2}\right)dr\\
:=&\widetilde{G}_{\text{up}}^{\alpha}(k)+\widetilde{G}_{\text{up}}^{\beta}(k).
\end{aligned}
\end{equation}
The former term can be analytically calculated,
\begin{equation}
\label{eq::e1up1}
\begin{split}
\widetilde{G}_{\text{up}}^{\alpha}(k)
&=2\pi\log b\sum_{\ell=M+1}^{\infty}s_\ell^2e^{-s_\ell^2k^2/4}
\end{split}
\end{equation}
where one uses the relation $w_\ell s_\ell=2\log b/\sqrt{\pi}$.

We next consider the estimation of $\widetilde{G}_{\text{up}}^{\beta}(k)$. Since the integral domain is $[0,r_c]$ with $r_c\ll s_M$ under Assumption ~\ref{ass:1}, one can use Taylor expansion to give this part a proper approximation. By using
\begin{equation}
e^{-r^2/s_\ell^2}=1+O\left(\frac{r^2}{s_\ell^2}\right), \quad \ell\geq M+1,\quad \text{and} \quad r<r_c
\end{equation}
one has
\begin{equation}\label{eq::e1up2}
\begin{split}
\widetilde{G}_{\text{up}}^{\beta}(k)&=\dfrac{4\pi}{k^2}\sum_{\ell=M+1}^{\infty}\omega_{\ell}\left[\frac{\sin(r_c k)}{k}-r_c\cos(kr_c)+O\left(k\int_{0}^{r_c}\sin(kr) \frac{r^3}{s_\ell^2}dr\right)\right]\\
&=\dfrac{4\pi}{k^2}\left[\frac{2\log b}{(b-1)\sqrt{\pi}s_M}+O\left(\frac{L^2}{s_{M}^3}\right)\right]\mathscr{P}(r_c,k).
\end{split}
\end{equation}
The proof of Proposition~\ref{prop::prop2} is completed by combining Eqs.~\eqref{eq::e1up1} and \eqref{eq::e1up2}.

\subsection{Proof of Proposition \ref{prop::prop3}}\label{pf::prop3}
For $\widetilde{G}_{\text{down}}$, the Fourier transform can be written explicitly,
\begin{equation}\label{eq::3.13}
\begin{split}
\widetilde{G}_{\text{down}}(k)=\dfrac{4\pi}{k}\sum_{\ell=-\infty}^{-1}\left[\frac{w_{\ell}s_{\ell}^2}{2}e^{-r_c^2/s_{\ell}^2}\sin(kr_c)+\frac{\sqrt{\pi}kw_{\ell}s_{\ell}^3}{4}e^{-s_{\ell}^2k^2/4}\Re\left\{\text{erfc}\left(\frac{r_c}{s_{\ell}}-\frac{ks_{\ell}}{2}i\right)\right\}\right],
\end{split}
\end{equation}
where $\Re\{\cdot\}$ denotes taking the real part. For further derivation, we need the following lemma. 
\begin{lemma}\label{lemma::erfc}
	The \emph{erfc}($\cdot$) function has an asymptotic expansion~\cite{abramowitz1964handbook,olver1997asymptotics}
	\begin{equation}
       \label{eq::ErfcExp}
	\emph{erfc}(z)= \frac{e^{-z^2}}{\sqrt{\pi}z}\left[\sum_{m=0}^{\mathcal{M}}(-1)^m\frac{(\frac{1}{2})_m}{z^{2m}}+O\left(z^{-2\mathcal{M}-2}\right)\right],~\text{as}~\Re\{z\}>1, 
	\end{equation}
	where $(\cdot)_{m}$ represents the Pochhammer's symbol.
\end{lemma}
Since $r_c/s_{\ell}>1$ for all $\ell\leq-1$, one can substitute Lemma~\ref{lemma::erfc} with $\mathcal{M}=0$ into Eq.~\eqref{eq::3.13} and then obtain
\begin{equation}
\widetilde{G}_{\text{down}}(k)=\frac{4\pi}{k^2}\sum_{\ell=-\infty}^{-1}w_{\ell}e^{-r_c^2/s_{\ell}^2}\beta_{\ell}(k),
\end{equation}
where the factor $\beta_{\ell}(k)$ is defined as Eq.~\eqref{eq::2.25}. It can be seen that $|\beta_{\ell}(k)|\leq r_c(1+O((r_c/s_{\ell})^{-2}))$ and $\beta_{\ell}(k)$ vanishes as $k\rightarrow 0$. One then completes the proof of Proposition~\ref{prop::prop3}. 

\section{Discussion on the integral transform \texorpdfstring{Eq.~\eqref{eq::approx}}~}\label{app::discussintegral}
Let us discuss the error involved by utilizing the integral transform given by Eq.~\eqref{eq::approx} in our estimation. In Eq.~\eqref{Uerr1} of Section~\ref{subsec::1}, this approximation is applied to $E_{T,m}(\bm{k})$. We have
\begin{equation}
\begin{split}
&\left|\sum_{\bm{k}\neq \bm{0}}E_{T,m}(\bm{k})-\dfrac{V}{(2\pi)^3}\int_{\frac{2\pi}{L}}^{\infty}k^2dk\int_{-1}^{1}d\cos\theta\int_{0}^{2\pi}d\varphi E_{T,m}(k,\theta,\varphi)\right|\\
\leq&\left|\sum_{\bm{k}\neq \bm{0}}E_{T,m}(\bm{k})-\frac{V}{(2\pi)^3}\int_{\mathbb{R}^3\backslash [-\frac{2\pi}{L},\frac{2\pi}{L}]^3}E_{T,m}(\bm{k})d\bm{k}\right|+\left(8-\frac{4\pi}{3}\right)\max_{\bm{k}\in[-\frac{2\pi}{L},\frac{2\pi}{L}]^3\backslash B_{2\pi/L}}|E_{T,m}(\bm{k})|\\
=&E_{T,m}^{\text{err}}+O\left((\log b)^{-1/2}e^{-\frac{m\pi^2}{2\log b}}V^{-1}\right)
\end{split}
\end{equation}
by the triangle inequality, where $B_{2\pi/L}:=\{\bm{k}|k\leq 2\pi/L\}$ and $E_{T,m}^{\text{err}}$ represents the error of the so-called ``punctured'' trapezoidal rule~\cite{izzo2022convergence}. Refining Theorem 2.1 and Lemma 2.4 in \cite{izzo2022convergence}, we find $E_{T,m}^{\text{err}}=O((\log b)^{-3/2}e^{-\frac{m\pi^2}{2\log b}}L^{-1})$. 
Similarly, we apply the integral transform to $(2V)^{-1}|\rho(\bm{k})|^2\widetilde{G}_{\text{up}}(k)$ and $(2V)^{-1}|\rho(\bm{k})|^2\widetilde{G}_{\text{down}}(k)$ in Eqs.~\eqref{eq::final2} and \eqref{eq::3.35}, respectively. Since 
\begin{equation}
\lim_{\bm{k}\rightarrow\bm{0}}\left|\rho(\bm{k})\right|^2\widetilde{G}_{\text{up}}(k)=\lim_{\bm{k}\rightarrow\bm{0}}\left|\rho(\bm{k})\right|^2\widetilde{G}_{\text{down}}(k)=0
\end{equation}
by the charge neutrality condition, one can add $\bm{k}=\bm{0}$ term in the discrete summation and replace the lower limit of integral from $2\pi/L$ to $0$. Also, by \cite{izzo2022convergence}, applying the integral transform to Eqs.~\eqref{eq::final2} and \eqref{eq::3.35} yields errors of $O(b^{-M}L^{-1})$ and $O(w_{-1}e^{-r_c^2/s_{-1}^2}L^{-1})$, respectively. Comparing these findings with those presented in Theorem~\ref{thm::thm4}, along with simple numerical tests, it is demonstrated that the integral transform maintains the leading decay order as $b\rightarrow 1$, even for relatively small $L$. For a practically used $L$, one can expect that the introduced error is typically one to two orders of magnitude smaller than the overall decomposition error. 

\section{Estimate of the force error}\label{app::Ferr}
We present the proof of force error $\bm{F}_{\text{err}}(\bm{r}_i)$ stated in Theorem~\ref{thm::thm4} here, which exhibits similarities with the energy estimate. Let us start by writing the Fourier spectral expansion of $\bm{F}_{\text{err}}(\bm{r}_i)$ provided in Eq.~\eqref{eq::UerrFerr} as the sum of three terms,
\begin{equation}\label{eq::Ferr1}
\begin{split}
\bm{F}_{\text{err}}(\bm{r}_i)=&\sum_{\bm{k}\neq \bm{0}}\frac{q_i\bm{k}}{V}\text{Im}\left(e^{-i\bm{k}\cdot\bm{r}_i}\rho(\bm{k})\right)\left[\widetilde{T}(k)+\widetilde{G}_{\text{up}}(k)+\widetilde{G}_{\text{down}}(k)\right]\\
:=&\bm{F}_{T}(\bm{r}_i)+\bm{F}_{G}^{\text{up}}(\bm{r}_i)+\bm{F}_{G}^{\text{down}}(\bm{r}_i).
\end{split}
\end{equation}
Note that all of $\widetilde{T}(k)$, $\widetilde{G}_{\text{up}}(k)$ and $\widetilde{G}_{\text{down}}(k)$ are radially symmetric, and one has
\begin{equation}
\text{Im}\left(e^{-i\bm{k}\cdot\bm{r}_i}\rho(\bm{k})\right)=\sum_{j=1}^{N}q_{j}\sin(-\bm{k}\cdot\bm{r}_{ij}).
\end{equation}
We use the spherical coordinates $(k,\theta,\varphi)$ such that the $z$-coordinate of $\bm{k}$ is in the direction of the vector $\bm{r}_{ij}$, it
will follow from the symmetry around the $z$-axis that the error of the term associated with $\bm{r}_{ij}$ is parallel to $\bm{r}_{ij}$. As a consequence, it is enough to calculate the  projection of the error of the force on the $z$-axis, i.e. to replace the vector extra $\bm{k}$ in Eq.~\eqref{eq::Ferr1} by $kz$ with $z=\cos\theta$. 

Let us consider the estimate of $\bm{F}_{T}(\bm{r}_i)$. By using Eq.~\eqref{eq::approx} and replacing $\bm{k}$ by $kz$, one has
\begin{equation}\label{eq::Ferr12}
\begin{split}
\bm{F}_{T}(\bm{r}_i)
\simeq \dfrac{1}{4\pi^2}\int_{\frac{2\pi}{L}}^{\infty}k^2\sum_{j=1}^{N}q_iq_j\int_{-1}^{1}kz\sin(-kr_{ij}z)dz\left[\sum_{m\neq 0}\widetilde{\phi}_m(k)\mathcal{C}(m)\right]dk.
\end{split}
\end{equation}
Similar to the case of $E_{T}$, one can divide the infinite integral into two parts,
\begin{equation}\label{eq::Ferr23}
	\bm{F}_{T}(\bm{r}_i)\simeq-\sum_{m\neq 0}\sum_{j=1}^{N}\frac{2\sqrt{2}q_iq_j}{\pi e^{\frac{m\pi^2}{2\log b}}}\left(\int_{\frac{2\pi}{L}}^{\frac{|\alpha_m|}{r_c}}+\int_{\frac{|\alpha_m|}{r_c}}^{\infty}\right)\dfrac{\mathscr{P}(r_{ij},k)}{r_{ij}^2}\left[\mathcal{I}_m-\mathcal{J}_{m}(k)\right]dk:=\bm{F}_{T,\mathcal{E}}+\bm{F}_{T,\mathcal{R}},
\end{equation}
where $\mathscr{P}(\cdot,\cdot)$ is defined as in Eq.~\eqref{eq::Prck}. Similar to Eqs.~\eqref{Uerr1}-~\eqref{eq::3.30}, it can be proved that
\begin{equation}\label{app::eqerr1}
\left|\bm{F}_{T,\mathcal{E}}\right|\simeq O\left((\log b)^{-3/2}e^{-\frac{\pi^2}{2\log b}}\right)\quad\text{and}\quad\left|\bm{F}_{T,\mathcal{R}}\right|\simeq O\left(\log b~e^{-\frac{\pi^2}{2\log b}}\right)\quad \text{as} \quad b\rightarrow 1
\end{equation}
by using the relation below,
\begin{equation}
	\dfrac{\mathscr{P}(r_{ij},k)}{r_{ij}^2}\rightarrow 0 \quad \text{as}\quad i=j\quad\text{or}\quad k\rightarrow0.
\end{equation}
One can obtain $\left|\bm{F}_{T}(\bm{r}_i)\right|\simeq O\left((\log b)^{-3/2}e^{-\pi^2/2\log b}\right)$ as $b\rightarrow 1$.

Next, one can obtain a reasonable estimate of $\bm{F}_{G}^{\text{up}}(\bm{r}_i)$ by finding the gradient of $E_{G}^{\text{up}}$ (using Eq.~\eqref{eq::final2}). Note that $\mathcal{T}(r_{ij},r_c)$ has zero derivative when $r_{ij}\neq r_c$, and has no definition at $r_{ij}=r_c$. One can simply take the continuous approximation and obtain 
\begin{equation}\label{eq::FGup}
\begin{split}
\bm{F}_{G}^{\text{up}}(\bm{r}_i)&\simeq\frac{2\log b}{\sqrt{\pi}}\sum_{j}\sum_{\ell=M+1}^{\infty}\frac{q_iq_j\bm{r}_{ij}}{s_{\ell}^3}e^{-r_{ij}^2/s_{\ell}^2}= O\left(\frac{1}{s_{M}^3}\right).\\
\end{split}
\end{equation}

Analogously, one can find
\begin{equation}\label{eq::AppC1}
	\bm{F}_{G}^{\text{down}}(\bm{r}_i)\simeq \dfrac{r_c}{4\pi }\sum_{\ell=-\infty}^{-1}\left[w_{\ell}e^{-r_c^2/s_{\ell}^2}\sum_{i,j}\frac{q_iq_j\bm{r}_{ij}}{r_{ij}^3}\left(1+\frac{2r_cr_{ij}}{s_{\ell}^2}\right)e^{-2r_c(r_{ij}-r_c)/s_{\ell}^2}H(r_{ij}-r_c)\right]
\end{equation}
by taking the negative gradient of Eq.~\eqref{eq::final3}. Since $r_c>s_{\ell}$ for all $\ell\leq-1$, it is naturally supposed that $2r_cr_{ij}/s_{\ell}^2\gg 1$ for $r_{ij}>r_c$. Together with $b\rightarrow 1$, one has 
\begin{equation}
\label{eq::force1}
	\left|\bm{F}_{G}^{\text{down}}(\bm{r}_i)\right|\simeq O\left(w_{-1}(s_{-1})^{-2}e^{-r_c^2/s_{-1}^2}\right).
\end{equation}
Finally, summarizing, one has the following estimate for the truncation error of force,
\begin{equation}
|\bm{F}_{\text{err}}(\bm{r}_i)|\simeq O\left((\log b)^{-3/2}e^{-\pi^2/2\log b}+b^{-3M}+w_{-1}(s_{-1})^{-2}e^{-r_c^2/s_{-1}^2}\right).
\end{equation}

\section{Ideal-gas assumption}
\label{Sec::Uncorrelation}
Let $\bm{S}$ be a statistic quantity  in an interacting particle and one attempts to analyze its root mean square value 
\begin{equation}
	\delta S:=\sqrt{\frac{1}{N}\sum_{i=1}^{N} \bm{S}_i^2}
\end{equation}
where $\bm{S}_i$ denotes the quantity acting on particle $i$ (energy for one dimension or force for three dimensions). Suppose that $\bm{S}_i$ has the following form 
\begin{equation}
   \bm{S}_i=q_i\sum_{j\neq i}q_j\bm{\zeta}_{ij}
\end{equation} 
due to the superposition principle of particle interactions, which indicates that the total effect on particle $i$ can be expressed as the sum of contributions for each $ij$ pair (including periodic images). Here $\bm{\zeta}_{ij}$ denotes the interaction between two particles. The ideal-gas assumption gives the following relation 
\begin{equation}
	\langle\bm{\zeta}_{ij}\bm{\zeta}_{ik}\rangle=\delta_{jk}\langle\bm{\zeta}_{ij}^2\rangle:=\delta_{jk}\zeta^2,
\end{equation}
where the expectation is taken over all particle configurations and $\zeta$ is  constant. The assumption describes that, every two different particle pairs are uncorrelated, and the variance of each pair is expected to be uniform. In computing the force variance of a charged system,  this assumption implies that
\begin{equation}
   \langle(\bm{S}_i)^2\rangle=q_i^2\sum_{j,k\neq i}q_jq_k\langle\bm{\zeta}_{ij}\bm{\zeta}_{ik}\rangle\approx q_i^2\zeta^2Q,
\end{equation}
where $Q=\sum_{i=1}^{N}q_i^2$. By the law of large numbers, one has $\delta S\approx \zeta Q/\sqrt{N},$ which can be used for the mean-field estimate of the force variance.

\end{document}